\documentclass[times,doublespace]{rncauth}

\usepackage{moreverb}

\usepackage[dvips,colorlinks,bookmarksopen,bookmarksnumbered,citecolor=red,urlcolor=red]{hyperref}

\newcommand\BibTeX{{\rmfamily B\kern-.05em \textsc{i\kern-.025em b}\kern-.08em
T\kern-.1667em\lower.7ex\hbox{E}\kern-.125emX}}

\def\volumeyear{2010}

\usepackage{wrapfig}

\usepackage{tikz}
\definecolor{grau}{rgb}{.9,0.9,0.9} % RGB 230,230,230
\definecolor{dunkelgrau}{rgb}{.7,0.7,0.7} % RGB 178,178,178
\definecolor{rot}{rgb}{0.89,0,0.10} % RGB 226,000,026

\newtheorem{lemma}{Lemma}[section]
\newtheorem{theorem}[lemma]{Theorem}
\newtheorem{proposition}[lemma]{Proposition}
\newtheorem{corollary}[lemma]{Corollary}
\newtheorem{definition}[lemma]{Definition}
\newtheorem{remark}[lemma]{Remark}

\newcommand{\R}{\ensuremath{\mathbb{R}}}

\newcommand{\V}{\ensuremath{\mathcal{V}}}
\newcommand{\E}{\ensuremath{\mathcal{E}}}

\newcommand{\Hnorm}{\ensuremath{\mathcal{H}}}
\newcommand{\id}{\ensuremath{\mbox{id}}}

\newcommand{\K}{\ensuremath{\mathcal{K}}}

\newcommand{\Kinf}{\ensuremath{\mathcal{K}_\infty}}

\begin{document}

\runningheads{S.~Dashkovskiy, M.~Kosmykov}{Reduction of the small gain condition for large-scale interconnections}

\title{Reduction of the small gain condition for large-scale interconnections}

\author{S.~Dashkovskiy\affil{1}\corrauth, M.~Kosmykov\affil{2}}

\address{\affil{1}University of Applied Sciences Erfurt, Department of Civil Engineering, Altonaer Straße 25, 99085 Erfurt, Germany\break
\affil{2}University of Bremen, Centre of Industrial Mathematics, P.O.Box 330440, 28334 Bremen, Germany}

\corraddr{E-mail: sergey.dashkovskiy@fh-erfurt.de}

\cgs{Michael Kosmykov is supported by the German Research Foundation (DFG) as part of the Collaborative Research Center 637 "Autonomous Cooperating Logistic Processes: A Paradigm Shift and its Limitations".}

\begin{abstract}
The small gain condition is sufficient for input-to-state stability (ISS) of interconnected
systems. However, verification of the small gain condition requires large amount of computations in
the case of a large size of the system. To facilitate this procedure we aggregate the subsystems and the gains between the subsystems
that belong to certain interconnection patterns (motifs) using three heuristic rules. These
rules are based on three motifs: sequentially connected nodes, nodes connected in parallel and almost
disconnected subgraphs. Aggregation of these motifs keeps the main structure of the mutual influences
between the subsystems in the network. Furthermore, fulfillment of the reduced small gain condition
implies ISS of the large network. Thus such reduction allows to decrease the number of computations
needed to verify the small gain condition. Finally, an ISS-Lyapunov function for the large network can be
constructed using the reduced small gain condition. Applications of these rules is illustrated on an example.
\end{abstract}


\maketitle

\vspace{-6pt}

\section{Introduction}
\vspace{-2pt}

Starting with the pioneering works \cite{JTP94}, \cite{JMW96}
 interconnection of input-to-state stable (ISS) systems has been studied by
many authors, see for example \cite{LaN03}, \cite{AnA07},
\cite{Cha05}, \cite{Ito08}. In particular, it is known that
cascades of ISS systems are ISS, while a feedback interconnection
of two ISS systems is in general unstable.  The first result of
the small gain type was proved in \cite{JTP94} for a feedback
interconnection of two ISS systems.  The Lyapunov version of this
result is given in \cite{JMW96}. These results were
generalized for an interconnection of $n\ge2$ systems in
\cite{DRW07}, \cite{DRW10}, \cite{JiW08} and \cite{KaJ09}. In particular, they state that if the so-called small gain condition $\Gamma(s)\not\geq s$ holds, then interconnection of ISS systems is ISS. Here the gain matrix $\Gamma$ describes an interconnection structure of the system.
 For recent results on the small gain conditions for a wider class of
interconnections we refer to \cite{JiW08}, \cite{ItJ07} and \cite{KaJ09}.
 In \cite{ItJ09} the authors consider necessary and sufficient small gain
 conditions for interconnections of two ISS systems in dissipative form.

Verification of the small gain condition  $\Gamma(s)\not\geq s$ requires large amount of analytical computations in the case of the large size  interconnected system.
 This procedure can be facilitated by reducing the size of the system, by applying a numerical method or applying a method that reduces the size of the gain matrix $\Gamma$ . 
  
Model reduction of linear large-scale systems is already a well-developed area. The most efficient approaches are balancing and moment matching (Krylov subspace methods), see \cite{Ant05}. In balancing methods state variables that are hard to control/observe are eliminated from the model. An approximation norm is usually given in terms of $\Hnorm_{\infty}$- or
$\Hnorm_2$-norms. In moment matching methods a function that matches certain moment of the Laurent series expansion is being looked for. 
These methods are computationally efficient in comparison with balancing methods, however provide no approximation error bounds. Usually, one uses a combination of both methods where first the large size is reduced by the moment matching methods and then the balancing method is applied.

On the contrary, the methods for the reduction of nonlinear systems are still in the development. As of today, there exist many different approaches that provide first steps in the direction of the reduction of nonlinear systems. However, these approaches are applied only to certain subclasses of nonlinear systems. The most known methods are an extension of the balancing and moment matching methods to nonlinear systems, proper orthogonal decomposition, singular perturbations theory, trajectory piecewise linear approach, Volterra methods and the theory of global attractors. The balancing methods \cite{Sch93}, \cite{LMG02} are applied to input-affine continuous-time nonlinear systems, and the moment matching to single-input single-output systems \cite{Ast10} and bilinear systems \cite{BrD10}. In the proper orthogonal decomposition (POD) \cite{HLB98}, \cite{Ast04}, \cite{HPS07} the original system is projected onto a subspace of a smaller dimension using the known set of data (snapshots). POD methods are usually applied to models describing physical systems. Singular perturbations theory \cite{BiA06}, \cite{KOS76} is used for the systems, where parameters evolve in different time scales ("slow" and "fast" parameters). This approach assumes aggregation of the variables evolving in the fast time scale. The trajectory piecewise linear approach \cite{ReW06} is mostly applied to input-affine systems. The system is linearized several times along a trajectory and the final model is constructed as a weighted sum of all local linearized reduced systems. In Volterra methods \cite{Phi03} the reduction is performed by taking into account the first several terms of the Polynomial expansion of a nonlinear function. In the theory of global attractors \cite{KKS10} one searches for a slow-manifold, inertial manifold or center manifold, on which a restricted dynamical system represents the "interesting" behaviour of the dynamical system.

Note that, if these methods will be directly applied to an application network, then information about the real physical objects of a network and of its structure will be, in general, lost. Therefore, a reduction method that preserves the main structure of the network is needed.

Structure preserving model reduction was studied in \cite{SVH08}, \cite{LKM03}, \cite{Sal05}. However, it is also applied only for particular classes of systems. 

Another possibility to decrease the number of analytical computations in verifying the small gain condition is an application of numerical methods \cite{Rueff07}, \cite{GeW11}. There is considered a local version of ISS. 

On the other hand,  to the best of our knowledge, there exist no approaches for the reduction of the size of the gain matrix $\Gamma$ in the small gain condition. In this paper we make the first attempt in this direction. On the other hand,  to the best of our knowledge, there exist no approaches for the reduction of the size of the gain matrix $\Gamma$ in the small gain condition. In this paper we make the first attempt in this direction. 
By the reduction we understand a reduction of the gain matrix, i.e. transition from the gain matrix $\Gamma$ of size $n$ to the matrix $\widetilde{\Gamma}$ of size $k<n$. 
To obtain the matrix $\widetilde{\Gamma}$ we propose to aggregate the subsystems and the gains between the subsystems that belong to certain interconnection patterns, so-called motifs \cite{motifs}. Aggregation of these motifs keeps the main structure of the mutual influences between the subsystems in the network. Thus the properties of the aggregated and the original models should be similar. This prompts us that ISS of the large-scale network can be established by checking the aggregated small gain condition corresponding to the smaller gain matrix $\widetilde{\Gamma}$.

In this paper we introduce three aggregation rules for the reduction of the gain matrix. These rules are based on three motifs: sequentially connected nodes, nodes connected in parallel and almost disconnected subgraphs. We establish that fulfillment of the reduced small gain condition implies ISS of the large network. Furthermore, we show how an ISS-Lyapunov function for the large network can be constructed using the so-called $\Omega$-path corresponding to the reduced small gain condition. 

In Section~\ref{sec:Notation} we introduce all necessary notation and recall the small gain condition. Then in Section~\ref{sec:Reduction} we show how the size of this condition can be reduced. This is illustrated on the example in  Section~\ref{sec:Example}. 
Conclusions are given in Section~\ref{sec:Conclusion}.

\section{Notation}\label{sec:Notation} 

\subsection{Vectors and spaces}

In the following we set ${\mathbb{R}}_{+}:=[0,\infty)$ and denote the
positive orthant ${\mathbb{R}}_{+}^n:= [0,\infty)^n$. The transpose of a
vector $x \in \mathbb{R}^n$ is denoted by $x^T$.  On ${\mathbb{R}}^n$ we
use the standard partial order induced by the positive orthant given by
\[
\begin{array}{l}
x\geq y\, \Longleftrightarrow\, x_i \geq y_i,\quad i=1,\ldots,n,\\
x>y\,\Longleftrightarrow\, x_i>y_i,\quad i=1,\ldots,n.
\end{array}\]
We write $x\not\geq y\,\Longleftrightarrow\, \exists
\,i\in\{1,\ldots,n\}: \, x_i<y_i.$ 

For a function $v:{\mathbb{R}}_{+}\to{\mathbb{R}}^{m}$ we
define its restriction to the interval $[s_1,s_2]$ by
\[v_{[s_1,\, s_2]}(t)=\left\{\begin{array}{ll}
v(t),& \mbox{if }t \in [s_1,s_2],\\
0,& \mbox{otherwise.}
\end{array}\right.\]

A continuous function $\alpha: \R_+\rightarrow\R_+$, where $\alpha(r)=0$ if and only if $r=0$, is called {\it positive definite}.
A function $\gamma:{\mathbb{R}}_{+}\to{\mathbb{R}}_{+}$ is said to be of
class $\cal{K}$ if it is continuous, strictly increasing and $\gamma(0)=0$. It
is of class $\cal{K}_{\infty}$ if, in addition, it is unbounded. Note that for
any $\alpha \in \cal{K}_{\infty}$ its inverse function $\alpha^{-1}$ always
exists and $\alpha^{-1}\in \cal{K}_{\infty}$.  A function $\beta:
{\mathbb{R}}_{+}\times{\mathbb{R}}_{+}\to{\mathbb{R}}_{+}$ is said to be
of class $\cal{KL}$ if, for each fixed $t$, the function $\beta(\cdot,t)$ is
of class $\cal{K}$ and, for each fixed $s$, the function $t\mapsto
\beta(s,t)$ is non-increasing and tends to zero for $t \to \infty$. By
$\id$ we denote the identity map.

Let $|\cdot|$ denote some norm in ${\mathbb{R}}^{n}$. The essential supremum
norm of a measurable function $\phi:\R_+\to \R^m$ is denoted by
${\|\phi\|}_{\infty}$. $L_{\infty}$ is the set of measurable functions for
which this norm is finite.

\subsection{Graphs}

We introduce also the notion of graphs from \cite{Ban09} and show how graphs can be described by matrices. A {\it directed graph with weights} consists of a finite {\it vertex set} $\V$ and an {\it edge set}
$\E$, where a {\it directed edge from vertex $i$ to vertex $j$} is an ordered pair
$(i,j)\in \E\subset \V \times \V$. The weights can be represented by a $|\V|\times
|\V|$ {\it weighted adjacency matrix} $A$, where $a_{ij}\geq 0$ denotes the weight of
the directed edge from vertex $i$ to vertex $j$. By convention $a_{ij}>0$, if
and only if $(i,j)\in \E$. We will denote a directed graph with weights of this
form by $G=(\V,\E,A)$. 
Additionally, we define
for each vertex $i$ the set of {\it successors} by
\begin{align}\label{graph_s}
S(i):=  \{j\,:\,(i,j)\in \E\}
\end{align}
and the set of {\it predecessors} by
\begin{align}\label{graph_p}
P(i):=  \{j\,:\,(j,i)\in \E\}.
\end{align}
A {\it path}\index{path} from vertex $i$ to $j$ is a sequence of distinct
vertices starting with $i$ and ending with $j$ such that there is a directed
edge between consecutive vertices. A directed graph is said to be {\it strongly connected}, if for any ordered pair $(i,j)$ of vertices, there is a
path which leads from $i$ to $j$. In terms of the weighted adjacency matrix
this is equivalent to the fact that $A$ is irreducible, \cite{BeP79}.

\subsection{Interconnected dynamical systems and input-to-state stability}

In this paper we study continuous dynamical systems. We consider a system
\begin{equation}\label{ws}
\dot{x}=f(x,u),\quad x \in {\mathbb{R}}^{n}, \quad
u\in{\mathbb{R}}^{m},
\end{equation}
and assume it  is forward complete, i.e., for all
initial values $x(0)\in {\mathbb{R}}^{n}$ and all essentially bounded
measurable inputs $u$ solutions $x(t) = x(t;x(0),u)$
exist for all positive times. Assume also that
for any initial value $x(0)$ and input $u$ the solution is unique.

We are interested in input-to-state stability (ISS) of systems of the form \eqref{ws}. We define ISS as follows, see \cite{Son89}. 
\begin{definition}[Input-to-state stability]
System of the form \eqref{ws} is called {\it input-to-state stable}
(ISS), if there exist functions $\beta\in\cal{KL}$ and
$\gamma\in\cal{K}$, such that for all $x(0)\in {\mathbb{R}}^{n}\,, u \in
L_{\infty}(\R_{+},\R^{m})$
\begin{eqnarray}\label{ws_iss}
|x(t)|\leq\max\{\beta(|x(0)|,t),\gamma({\|u\|}_{\infty})\}\,,\quad t\ge0.
\end{eqnarray}
\end{definition}

We assume that there are $n$ interconnected dynamical systems given by
\begin{equation}\label{is}
\begin{array}{ccc}
{\dot{x}}_1&=&f_1(x_1,\ldots,x_n,u_1)\\
&\vdots&\\
{\dot{x}}_n&=&f_n(x_1,\ldots,x_n,u_n)
\end{array}
\end{equation} where $x_i \in {\mathbb{R}}^{N_i}$, $u_i \in
{\mathbb{R}}^{m_i}$ and the functions $f_i:
{\mathbb{R}}^{\sum_{j=1}^{n}N_j +
m_i}\rightarrow{\mathbb{R}}^{N_i}$ are continuous and for all $r
\in {\mathbb{R}}$ are locally Lipschitz continuous in
$x={({x_1}^T,\ldots,{x_n}^T)}^T$ uniformly in $u_i$ for $|u_i|\leq
r$. This regularity condition for $f_i$ guarantees the existence
and uniqueness of solution for the $i$th subsystem for a given
initial condition and input $u_i$.

The interconnection \eqref{is} can be written as
\eqref{ws} with $x:=(x_1^T,\dots,x_n^T)^T$,\\
$u:=(u_1^T,\dots,u_n^T)^T$  and
\[f(x,u)=\left(f_1(x_1,\dots,x_n,u_1)^T,\ldots,
f_n(x_1,\dots,x_n,u_n)^T\right)^T.\]
If we consider the individual subsystems, we treat the state $x_j,  j
\neq i$ as an independent input for the $i$th subsystem.

The $i$th subsystem of \eqref{is} is ISS, if
there exist functions
$\beta_i$ of class $\cal{KL}$, $\gamma_{ij}$, $\gamma_i \in
\cal{K}_\infty\cup\{$0$\}$ such that for all initial values
$x_i(0)$ and inputs $u \in L_{\infty}(\R_{+},\R^{m})$ there exists a unique
solution $x_i(\cdot)$ satisfying for all $t\ge0$
\begin{equation}\label{is_iss}
|x_i(t)|\hspace{-0.1cm}\leq\hspace{-0.1cm}\max\{\beta_i(|x_i(0)|,t),\hspace{-0.05cm}\max\limits_j\{\gamma_{ij}({\|x_{j[0,t]}\|}_{\infty})\},\hspace{-0.05cm}\gamma_i({\|u\|}_{\infty})\}
\end{equation}

Another notion useful for stability investigations of nonlinear systems
is the notion of an ISS-Lyapunov function. 
\begin{definition}[ISS-Lyapunov function]
Continuous function $V:\R^n\rightarrow\R_{+}$ is called an ISS Lyapunov function for $\Sigma$, if it is locally Lipschitz continuous on $\R^n\backslash\{0\}$ and
\begin{itemize}
\item $\exists\psi_1,\psi_2\in\K_{\infty}$ such that
\[\psi_1(\|x\|)\leq V(x)\leq\psi_2(\|x\|), \quad \forall x\in\R^n.\]
\item $\exists\gamma\in\K$, and positive definite function $\alpha$ such that in all points of differentiability of $V$
\begin{equation*}
V(x)\geq\gamma(\|u\|)\Rightarrow \nabla V(x)f(x,u)\leq-\alpha(\|x\|).
\end{equation*}
\end{itemize}
\end{definition}
\begin{remark}
In Theorem~2.3 in \cite{DRW10} it was proved that the system
\eqref{ws} is ISS if and only if it admits an (not
necessarily smooth) ISS Lyapunov function.
\end{remark}
A locally Lipschitz continuous function $V_i:\R^N_i\rightarrow\R_{+}$ is an ISS Lyapunov function for \eqref{is}, if:
\begin{itemize}
\item $\exists\psi_1,\psi_2{\in}\K_{\infty}:\quad\psi_{i1}(\|x_i\|){\leq} V_i(x_i){\leq}\psi_{2i}(\|x_i\|),  \forall x_i{\in}\R^{n_i}.$
\item$\exists\gamma_{ij}{\in}\K_{\infty}\cup\{0\}$, $j{\neq }i$, $\exists\gamma_i{\in}\K$: in all points of dif. of $V_i$
\begin{equation}\label{iss_lyap}
V_i(x_i)\geq\max\{\max_{j,j\neq i}\gamma_{ij}(V_j(x_j)),\gamma_i(\|u\|)\}\Rightarrow
\nabla V_i(x_i)f_i(x,u)\leq-\alpha_i(\|x_i\|).
\end{equation}
\end{itemize}

Note that an interconnection of subsystems of the form \eqref{is} can be unstable, i.e., not ISS, even if each of its subsystems is ISS. In the following subsection we recall known conditions that guarantee stability for interconnections of ISS systems.

\subsection{Known stability results for interconnected systems}\label{sec:SGC} 

To establish ISS of an interconnected system of the form \eqref{ws} we collect the gains $\gamma_{ij}\in \Kinf\cup \{0\}$ of the ISS conditions
\eqref{is_iss}  in a matrix
$\Gamma=(\gamma_{ij})_{n\times n}$, with the convention $\gamma_{ii}\equiv
0$,\, $i=1,\dots,n$. The operator $\Gamma:\R_{+}^n\rightarrow\R_{+}^n$ is then
defined by
 \begin{equation}\label{operator_gamma}
\Gamma(s):=\left(\begin{array}{c}\max\{\gamma_{1,2}(s_2),\dots,\gamma_{1,n}(s_n)\}\\
\vdots\\
\max\{\gamma_{n1}(s_1),\dots,\gamma_{n,n-1}(s_{n-1})\}
\end{array}\right)\,.
\end{equation}

In \cite{DRW07} the following theorem was proved that establishes ISS of interconnected system \eqref{ws}.
\begin{theorem}[Small gain theorem]\label{thm:iss}
Assume that each subsystem \eqref{is} is ISS. If condition
\begin{equation}\label{sgc_max}
\Gamma(s) \not\geq s, \forall s \in {\R}_{+}^n\backslash\{0\}
\end{equation}
holds, then \eqref{ws} is ISS.
\end{theorem}
The condition $\Gamma(s) \not\geq s$ is called {\it small gain condition}.

A Lyapunov type counterpart of the small gain theorem was proved in \cite{DRW10}. To recall this result we need first the following notion of  
an $\Omega$-path, see \cite{DRW10}.
\begin{definition}[$\Omega$-path]
A continuous path $\sigma\in\K_{\infty}^n$ is called an $\Omega$-path
with respect to $\Gamma$ if
\begin{enumerate}
\item for each $i$, the function $\sigma^{-1}_i$ is locally Lipschitz
  continuous on $(0,\infty)$;
\item for every compact set $P\subset(0,\infty)$ there are finite
constants $0<c<C$ such that for all points of differentiability of
$\sigma^{-1}_i$ and $i=1,\ldots,n$ we have
\begin{equation}\label{sigma cond 1}
0<c\leq(\sigma^{-1}_i)'(r)\leq C, \quad \forall r\in P
\end{equation}
\item for all $r>0$ it holds that 
\begin{equation}\label{sigma cond 2'}
\Gamma(\sigma(r))\leq\sigma(r).
\end{equation}
\end{enumerate}
\end{definition}
\begin{remark}\label{omega_max}
In \cite[Proposition~2.7]{KaJ09} and \cite[Proposition~2.3.14]{Kos11} it was shown that if the small gain condition $\Gamma(s) \not\geq s$, $\forall s \in
{\R}_{+}^n\backslash\{0\}$ holds, then $\Omega$-path can be constructed as $\sigma(t):=Q(at)$ for some vector $a>0$ and $Q(s):=(Q_1(s),\ldots,Q_n(s))^T$, where $Q_i(s):=\max \{s_i,(\Gamma(s))_i,\ldots,(\Gamma^{n-1}(s))_i\}$.
\end{remark}

\begin{theorem}[Small gain theorem (in terms of Lyapunov functions)]\label{thm:iss_lyap}
  Assume that each subsystem of \eqref{is} has an ISS Lyapunov
  function $V_i$ and the corresponding gain operator $\Gamma$ is given by
  \eqref{operator_gamma}. If $\Gamma(s) \not\geq s$   for all $s\neq 0, s\geq 0$ is satisfied, then the system \eqref{ws}
   is ISS and an ISS Lyapunov function is given by
\begin{equation}\label{Lyapunov_function}
V(x)=\max\limits_{i=1,\ldots,n}\sigma_i^{-1}(V_i(x_i)),
\end{equation}
where $\sigma\in\K_{\infty}^n$ is an arbitrary $\Omega$-path with respect to $\Gamma$.
\end{theorem}

Consider an interconnected system of the form \eqref{ws}. Assume, for convenience that all its subsystems are ISS with gains $\gamma_{ij}$ collected in the gain matrix $\Gamma$. Then, to establish ISS of the interconnection we can use Theorem~\ref{thm:iss}, i.e. we need to verify the small gain condition \eqref{sgc_max}. By \cite[Lemma~2.3.14]{Rueff07} the small gain condition \eqref{sgc_max} is equivalent to the cycle condition
\begin{equation}\label{cycle}
\gamma_{k_1k_2} \circ \gamma_{k_2k_3} \circ \dots \circ \gamma_{k_{r-1}k_r} < \id,
\end{equation}
for all $(k_1,...,k_r) \in \{1,...,n\}^r$ with $k_1=k_r$. The largest possible number of cycles to be checked in this condition can be calculated 
as $\sum_{k=2}^n\left(\hspace{-0.2cm}
\begin{array}{l}
n \\
k
\end{array}\hspace{-0.2cm}\right)k!$, where $\left(\hspace{-0.2cm}\begin{array}{l}
n \\
k
\end{array}\hspace{-0.2cm}\right)$ is the binomial coefficient. 

Thus, too many scalar inequalities need to be verified on $\R_+$ in case of large $n$.

\section{Reduction rules}\label{sec:Reduction}

To reduce the size of the gain matrix in the small gain condition \eqref{sgc_max} we model the structure of the network described in \eqref{ws} as a directed graph with weights $G=(\V,\E,\Gamma)$. The vertex set $\V=\{1,\dots,n\}$ corresponds to the subsystems of the network, the edge set $\E$ to the interconnection between subsystems, i.e.
\begin{equation}
e_{ij}=\left\{
\begin{array}{ll}
1,&\mbox{ if }\gamma_{ij}\not\equiv 0,\\
0,&\mbox{otherwise}.
\end{array}
\right.
\end{equation}
The weight of the edge $e_{ij}$ from vertex $i$ to $j$  is given by $\gamma_{ji}$ and describes the influence of subsystem $i$ on subsystem $j$. All the weights are collected in the gain matrix $\Gamma$. Note that the matrix $\Gamma$ is not static, i.e. the weights are in general nonlinear functions. 

In our reduction approach we propose to reduce the size of the gain matrix $\Gamma$ in the small gain condition \eqref{sgc_max}. In particular, we transform the graph $G=(\V,\E,\Gamma)$ 
by introducing aggregation rules for vertices for typical subgraphs occurring in the network. Such subgraphs we will call \textit{motifs}\index{motifs}  \cite{motifs}. By aggregation of the vertices we understand the construction of a smaller graph $\widetilde G=(\widetilde \V,\widetilde \E,\widetilde \Gamma)$ in which the vertices may represent nonempty subsets of vertices in the original graph $G=(\V,\E,\Gamma)$. 
We single out the following motifs: parallel connections, sequential connections of vertices and almost disconnected subgraphs. These reduction rules are inspired by the properties of motifs in \cite{AlT05}.

\subsection{Aggregation of sequentially connected nodes}
The vertices of the set ${\V}_J= \{v_1,...,v_k \}$ are called {\it
  sequentially connected}, see Figure~\ref{fig:ser_large_general},  if there exist vertices $v,v' \in \V \setminus {\V}_J$
such that
\begin{align*}
P(v_i)=
\begin{cases}
v &\quad i=1,\\
v_{i-1} &\quad i=2,...,k
\end{cases}
\end{align*}
and 
\begin{align*}
S(v_i)=
\begin{cases}
v_{i+1} &\quad i=1,...,k-1,\\
v'&\quad i=l.
\end{cases}
\end{align*}
The predecessor set $P$ and successor set $S$ were defined in \eqref{graph_p} and  \eqref{graph_s}.

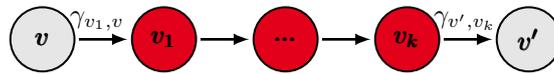
\begin{figure}[tbh]
\centering
\begin{tikzpicture}[scale=0.8] %Grafik 5
\begin{scope}[>=latex]
\filldraw[fill=grau,line width=0.75pt] (0,0) circle (15pt) node {$\pmb{v}$};
\filldraw[fill=rot,line width=0.75pt] (2,0) circle (15pt) node {$\pmb{v_1}$};
\filldraw[fill=rot,line width=0.75pt] (4,0) circle (15pt) node {$\pmb{...}$};
\filldraw[fill=rot,line width=0.75pt] (6,0) circle (15pt) node {$\pmb{v_k}$};
\filldraw[fill=grau,line width=0.75pt] (8,0) circle (15pt) node {$\pmb{v'}$};
\draw[line width=1pt] (0.9,0.3) node {\textbf{$\gamma_{v_1,v}$}};
\draw[line width=1pt] (6.95,0.3) node {\textbf{$\gamma_{v',v_k}$}};
	
\draw[->,line width=1pt] (0.6,0) -- (1.4,0); 			%V->V1
\draw[->,line width=1pt] (2.6,0) -- (3.4,0); 			%V1->...
\draw[->,line width=1pt] (4.6,0) -- (5.4,0); 			%...->Vk
\draw[->,line width=1pt] (6.6,0) -- (7.4,0); 			%Vk->V'
\end{scope}
\end{tikzpicture}
\caption{Sequential connection of vertices $v_1,\ldots,v_k$.}
\label{fig:ser_large_general} 
\end{figure}

The corresponding gain matrix is given by
\begin{equation}\label{gamma_ser}
\Gamma=\left(
\begin{array}{ccccccccc}
\ldots&\ldots&\ldots&\ldots&\ldots&\ldots&\ldots&\ldots&\ldots\\
\ldots&\ldots&0&0&\ldots&0&0&\ldots&\ldots\\
\ldots&0&\gamma_{v_1,v}&0&\ldots&0&0&0&\ldots\\
\ldots&0&0&\gamma_{v_2,v_1}&\dots&0&0&0&\ldots\\
\ldots&\vdots&\vdots&\vdots&\ddots&\vdots&\vdots&\vdots&\ldots\\
\ldots&0&0&0&\ldots&\gamma_{v_k,v_{l-1}}&0&0&\ldots\\
\ldots&\ldots&\ldots&0&\ldots&0&\gamma_{v',v_k}&\ldots&\ldots\\
\ldots&\ldots&0&0&\ldots&0&0&\ldots&\ldots\\
\ldots&\ldots&\ldots&\ldots&\ldots&\ldots&\ldots&\ldots&\ldots\\
\end{array}
\right).
\end{equation}
The cycle condition  \eqref{cycle} for the cycles that include nodes from $\{v_1,\dots,v_k\}$  looks as follows:
\begin{equation}\label{cycle_large_ser}
\ldots\circ\gamma_{v',v_{k}}\circ\ldots\circ\gamma_{v_2,v_1}\circ\gamma_{v_1,v}\ldots<\id.
\end{equation}
\subsubsection{Aggregation of gains}\quad\\
To obtain a graph of a smaller size we aggregate the nodes $v_1,\dots,v_k$ with the node $v$. We denote the new vertex by $J$. A cut-out of the new reduced graph is shown in Figure~\ref{fig:ser_small_general}. So, we consider the reduced graph $\widetilde{G}=(\widetilde{\V},\widetilde{\E},\widetilde{\Gamma})$, where the vertices are given by
\begin{align}\label{seq:tildeV}
 \tilde \V=(\V\setminus (\V_J\cup \{v\}) ) \cup J
\end{align}
and the edges are given by
\begin{multline}\label{seq:tildeE}
\tilde \E= \E\,  \setminus \, (\{(v,w),(w,v'),(w_1,w_2) : w,w_1,w_2 \in \V_J\}\cup(v,v'))\, \\ \cup\,\{(J,v')\cup(u,J): \, (u,v)\in \E\}.
\end{multline}

The corresponding weighted adjacency matrix $\tilde \Gamma$ of the dimension $n-k$ can be obtained from $\Gamma$,
where the rows and columns corresponding to the vertices $v,v_1,\ldots,v_k$ are
replaced by a row and a column corresponding to the new vertex $J$. The weights are
then given by  
\begin{eqnarray}\label{gain_ser}
\widetilde{\gamma}_{v',J}:=\max\{\gamma_{v',v_k}\circ\dots\circ\gamma_{v_2,v_1}\circ\gamma_{v_1,v},\gamma_{v',v}\},
\end{eqnarray}
\begin{eqnarray}\label{gain_ser2}
\widetilde{\gamma}_{J,v'}:=\gamma_{v,v'},\quad \widetilde{\gamma}_{J,j}:=\gamma_{v,j},\quad \widetilde{\gamma}_{j,J}:=\gamma_{j,J},\quad j\in \V\setminus(\V_J\cup\{v,v'\}).
\end{eqnarray}

\begin{figure}[tbh]
\centering
\begin{tikzpicture}[scale=0.8] %Grafik 6
\begin{scope}[>=latex]
\filldraw[fill=grau,line width=0.75pt] (0,0) circle (15pt) node {$\pmb{J}$};
\filldraw[fill=grau,line width=0.75pt] (4,0) circle (15pt) node {$\pmb{v'}$};
\draw[line width=1pt] (2,0.5) node {\textbf{$\widetilde\gamma_{v',J}$}};
\draw[->,line width=1pt] (0.6,0) -- (3.4,0); 			%V->J
\end{scope}
\end{tikzpicture}
\caption{Vertices $v_1,\ldots,v_k,v'$ are aggregated.}
\label{fig:ser_small_general}  
\end{figure}
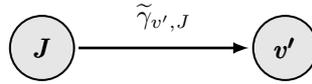

Other gains stay the same, i.e. 
\begin{eqnarray}\label{gain_ser3}
\widetilde{\gamma}_{ij}:=\gamma_{ij}, \,i,j\neq J.
\end{eqnarray}
The small gain condition \eqref{sgc_max} corresponding to the reduced gain matrix $\tilde \Gamma$ has the following properties.
\begin{theorem}\label{theorem_seq}
Consider a gain matrix $\Gamma$ as in \eqref{gamma_ser} where the nodes $\{v_1,\ldots,v_k\}$ of the corresponding graph are sequentially connected.
Then condition \eqref{sgc_max} holds for the matrix $\Gamma$ if and only if condition \eqref{sgc_max} holds for aggregated the matrix $\tilde \Gamma$ with gains defined in \eqref{gain_ser}-\eqref{gain_ser3}. 

Assume that there were $p$ cycles that include one of the nodes  $v_i$ from $V_J$. If $\gamma_{v',v}\neq 0$, then the number of cycles to be checked in the cycle condition \eqref{cycle} corresponding to the reduced matrix $\tilde \Gamma$ is decreased by $p$  after the aggregation, otherwise it stays the same.
\end{theorem}
\begin{proof}
Let condition \eqref{sgc_max} for the gains defined in \eqref{gain_ser}-\eqref{gain_ser3}  hold. Then the cycle condition \eqref{cycle} corresponding to these gains holds. In particular, for the cycles containing the gain $\widetilde{\gamma}_{v',J}$ the following inequality holds:
\begin{equation}\label{cycle_small_ser}
\ldots\circ\widetilde{\gamma}_{v',J}\circ\ldots<\id.
\end{equation}
From the definition of the gain $\widetilde{\gamma}_{v',J}$ in \eqref{gain_ser} condition \eqref{cycle_large_ser} holds. Condition \eqref{cycle} on other cycles corresponding to $\Gamma$ is satisfied straightforwardly. Thus, the matrix $\Gamma$ satisfies \eqref{sgc_max}.

For the proof in the opposite direction we consider the cycle condition \eqref{cycle_large_ser} and from  \eqref{gain_ser} obtain \eqref{cycle_small_ser}, i.e. cycle condition for the aggregated gain $\widetilde{\gamma}_{v',J}$. The rest cycle conditions are satisfied straightforwardly. Thus, the matrix $\tilde\Gamma$ satisfies \eqref{sgc_max}.

If $\gamma_{v',J}= 0$, then, as the cycle containing one of the nodes $\{v_1,\dots,v_k\}$ contains necessarily all other nodes from $\{v_1,\dots,v_k\}$, the number of cycles to be checked in the cycle condition is the same. Otherwise, these cycles will "coincide" with the cycles that include gain $\gamma_{v',v}$. Thus, the overall number of the cycles will decrease by $p$. 
\end{proof}
Thus, to show that a system of the form \eqref{ws} is ISS, it is enough to verify the small gain condition  $\widetilde{\Gamma}(s)\not\geq s$ corresponding to the reduced gain matrix  $\widetilde{\Gamma}$.
\begin{corollary}\label{cor:seq}
Consider interconnected system \eqref{ws} and assume that  all the subsystems in \eqref{is} are ISS with gains as in \eqref{is_iss}. If condition \eqref{sgc_max} holds for the gains defined in \eqref{gain_ser}-\eqref{gain_ser3}, then the system \eqref{ws} is ISS.
\end{corollary}
\begin{proof}
The assertion follows from Theorem~\ref{theorem_seq} and Theorem~\ref{thm:iss}.
\end{proof}
\subsubsection{Construction of an $\Omega$-path}\quad\\
To construct an ISS-Lyapunov function of the interconnected system \eqref{ws}, we can apply Theorem~\ref{thm:iss_lyap}. However, for this purpose we need to have an $\Omega$-path $\sigma$ satisfying  \eqref{sigma cond 2'}, i.e.
\[\Gamma(\sigma)\leq\sigma.\]
It appears, that if an $\Omega$-path corresponding to the reduced gain matrix $\widetilde\Gamma$ is known, we can calculate an $\Omega$-path for the large gain matrix $\Gamma$. Furthermore, we can additionally show that knowing  an $\Omega$-path for large gain matrix $\Gamma$ we can construct  an $\Omega$-path for the small gain matrix $\Gamma$
\begin{proposition}\label{proposition_red_seq}
Consider a gain matrix $\Gamma$ and the corresponding reduced gain matrix $\tilde \Gamma$ with gains defined in \eqref{gain_ser}-\eqref{gain_ser3}. Then:\\
(i)\quad If an $\Omega$-path $\widetilde{\sigma}$ for $\tilde \Gamma$  satisfying \eqref{sigma cond 2'} is given, then an $\Omega$-path $\bar{\sigma}$ for the matrix $\Gamma$ can be constructed as 
\begin{equation}\label{omega_ser}
\bar{\sigma}_{w}:=
\left\{
\begin{array}{ll}
\gamma_{v_i,v_{i-1}}\circ\gamma_{v_{i-1},v_{i-2}}\circ\dots\circ\gamma_{v_2,v_{1}}\circ\gamma_{v_1,v}\circ\widetilde{\sigma}_J,&\mbox{ if } w=v_i, i\in\{1,\dots,k\},\\
\widetilde{\sigma}_w,&\mbox{ otherwise };
\end{array}
\right.
\end{equation}
(ii)\quad  If an $\Omega$-path $\bar{\sigma}$ for $ \Gamma$  satisfying \eqref{sigma cond 2'} is given, then an $\Omega$-path $\tilde{\sigma}$ for the matrix $\tilde\Gamma$ can be constructed as 
\begin{equation}\label{omega_ser_2}
\bar{\sigma}_{w}:=
\left\{
\begin{array}{ll}
\bar{\sigma}_w &\mbox{ if } w\in\tilde \V\setminus J,\\
\bar{\sigma}_v,&\mbox{ if } w=J;
\end{array}
\right.
\end{equation}
\end{proposition}
\begin{proof}\quad
\\
{\it Proof of (i)}:\\
We assume that an $\Omega$-path $\widetilde{\sigma}$ for the small gain matrix $\tilde \Gamma$ is known. In particular, by \eqref{sigma cond 2'}
 $\widetilde{\Gamma}(\widetilde{\sigma})\leq\widetilde{\sigma}$ holds.  Let us check whether $\Omega$-path $\bar{\sigma}$ defined in \eqref{omega_ser} is an  $\Omega$-path for the large gain matrix $\Gamma$. To this end we need to check \eqref{sigma cond 2'} for $\bar{\sigma}$.
 
 For the components $\Gamma(\bar{\sigma})_w$, $w\not\in\{v_1,\dots,v_k,v'\}$ the inequality  \eqref{sigma cond 2'} holds straightforwardly.  Consider now  $\Gamma(\bar{\sigma})_w$, $w=v_i, i\in\{1,\dots,k\}$. Applying \eqref{gain_ser}-\eqref{gain_ser3} and \eqref{omega_ser} we obtain:
\begin{eqnarray*}
\Gamma(\bar{\sigma})_{v_i}&=&\gamma_{v_i,v_{i-1}}\circ\bar{\sigma}_{v_{i-1}}=\gamma_{v_i,v_{i-1}}\circ\gamma_{v_{i-1},v_{i-2}}\circ\bar{\sigma}_{v_{i-2}}\\
&=&\dots=\gamma_{v_i,v_{i-1}}\circ\dots\circ\gamma_{v_{1},v}\circ\widetilde{\sigma}_{J}\\
&=&\bar{\sigma}_{v_i};\\
\Gamma(\bar{\sigma})_{v'}&=&\max\{\gamma_{v',1}(\bar{\sigma}_1),\dots,\gamma_{v',v_k}(\bar{\sigma}_{v_k}),\dots,\gamma_{v',n}(\bar{\sigma}_{n})\}\\
&=&\max\{\widetilde{\gamma}_{v',1}(\widetilde{\sigma}_{1}),\dots,\underbrace{\gamma_{v',v_k}\circ\dots\circ\gamma_{v_1,v}\circ\widetilde{\sigma}_J
}_{\widetilde{\gamma}_{v',J}\circ\widetilde{\sigma}_J},\dots,\widetilde{\gamma}_{v',{n}}(\widetilde{\sigma}_{n})\}\\
&=&\max\{\widetilde{\gamma}_{v',1}(\widetilde{\sigma}_{1}),\dots,\widetilde{\gamma}_{v',J}\circ\widetilde{\sigma}_J,\dots,\widetilde{\gamma}_{v',n}(\widetilde{\sigma}_{n})\\
&\leq&\widetilde{\sigma}_{v'}=\bar{\sigma}_{v'}.
\end{eqnarray*}
Thus  $\Gamma(\bar{\sigma})\leq\bar{\sigma}$ and $\bar{\sigma}$ is an $\Omega$-path corresponding to the large gain matrix $\Gamma$.\\
{\it Proof of (ii)}:\\
Assume now that an $\Omega$-path $\bar{\sigma}$ for the large gain matrix $\Gamma$ is known. Let us check whether $\Omega$-path $\tilde{\sigma}$ defined in \eqref{omega_ser_2} is an  $\Omega$-path for the small gain matrix $\tilde\Gamma$. To this end we need to check \eqref{sigma cond 2'} for $\tilde{\sigma}$.
 
For the components $\tilde\Gamma(\tilde{\sigma})_w$, $w\neq v'$ the inequality  \eqref{sigma cond 2'} holds straightforwardly.  Consider now  $\tilde\Gamma(\tilde{\sigma})_{v'}$. Applying \eqref{gain_ser}, \eqref{omega_ser_2} and \eqref{sigma cond 2'} for $w\neq v'$ we obtain:
\begin{eqnarray*}
\tilde\Gamma(\tilde{\sigma})_{v'}&=&\max\{\tilde\gamma_{v',1}\circ\tilde{\sigma}_1,\ldots,\tilde\gamma_{v',J}\circ\tilde{\sigma}_J,\ldots,\tilde\gamma_{v',n}\circ\tilde{\sigma}_n\}\\
&=&\max\{\tilde\gamma_{v',1}\circ\tilde{\sigma}_1,\ldots,\max\{\gamma_{v',v_k}\circ\dots\circ\gamma_{v_2,v_1}\circ\gamma_{v_1,v},\gamma_{v',v}\}\circ\bar{\sigma}_v,\ldots,\tilde\gamma_{v',n}\circ\tilde{\sigma}_n\}\\
&\leq&\max\{\tilde\gamma_{v',1}\circ\tilde{\sigma}_1,\ldots,\gamma_{v',v_k}\circ\dots\circ\gamma_{v_2,v_1}\circ\bar{\sigma}_{v_1},\bar\sigma_{v'},\ldots,\tilde\gamma_{v',n}\circ\tilde{\sigma}_n\}\\
&&\vdots\\
&\leq&\max\{\tilde\gamma_{v',1}\circ\tilde{\sigma}_1,\ldots,\gamma_{v',v_k}\circ\bar{\sigma}_{v_k},\bar\sigma_{v'},\ldots,\tilde\gamma_{v',n}\circ\tilde{\sigma}_n\}\\
&\leq&\max\{\gamma_{v',1}\circ\bar{\sigma}_1,\ldots,\bar\sigma_{v'},\bar\sigma_{v'},\ldots,\gamma_{v',n}\circ\bar{\sigma}_n\}\\
&\leq&\bar\sigma_{v'}=\tilde\sigma_{v'}
\end{eqnarray*}
Thus  $\tilde\Gamma(\tilde{\sigma})\leq\tilde{\sigma}$ and $\tilde{\sigma}$ is an $\Omega$-path corresponding to the small gain matrix $\tilde\Gamma$.
\end{proof}
The proposition above implies the following result concerning the construction of an ISS-Lyapunov function.
\begin{corollary}
Consider a system of the form \eqref{ws} that is interconnection of subsystems
\eqref{is}. Assume that each subsystem {\it i} of \eqref{is} has
an ISS-Lyapunov function $V_i$ with the corresponding ISS-Lyapunov
gains $\gamma_{ij}, \gamma_i, i,j=1,\ldots,n$ as in \eqref{iss_lyap}. If \eqref{sgc_max} holds for $\widetilde\Gamma$ defined by \eqref{gain_ser}-\eqref{gain_ser3}. Then the system \eqref{ws} is has an ISS-Lyapunov function and an ISS-Lyapunov function is given by \eqref{Lyapunov_function} with $\sigma$ from \eqref{omega_ser}.
\end{corollary}
\begin{proof}
The assertion follows from Theorem~\ref{thm:iss_lyap} and Proposition~\ref{proposition_red_seq}.
\end{proof}

\subsection{Aggregation of nodes connected in parallel}
Parallel connections are characterized by the vertices having the same predecessor
and successor sets consisting of a single vertex. Let the vertices $\V_J:=\{v_1,\ldots,v_k\} \subset \V$ be \textit{connected in parallel}, i.e. every
vertex has only one ingoing and one outgoing edge and the ingoing edges
originate from one vertex $v \in \V$ and also the outgoing edges end in solely
one vertex $v'\in \V$, see Figure~\ref{fig:par_large_general}. To be precise, $\V_J = \{ i \in \V \, : P(i)=v,
S(i)=v'\}$.

\begin{figure}[tbh]
\centering
\begin{tikzpicture}[scale=0.8] 
\begin{scope}[>=latex]
\filldraw[fill=grau,line width=0.75pt] (2.3,3.6) circle (15pt) node {$\pmb{v}$};
\filldraw[fill=rot,line width=0.75pt] (1,1) circle (15pt) node {$\pmb{v_1}$};
\filldraw[fill=rot,line width=0.75pt] (3.6,1) circle (15pt) node {$\pmb{v_k}$};
\filldraw[fill=grau,line width=0.75pt] (2.3,-1.6) circle (15pt) node {$\pmb{v'}$};
\draw[line width=1pt] (2.3,1) node {\textbf{...}};
\draw[line width=1pt] (1,2.7) node {\textbf{$\gamma_{v_1,v}$}};
\draw[line width=1pt] (3.5,2.7) node {\textbf{$\gamma_{v_k,v}$}};
\draw[line width=1pt] (1,-0.4) node {\textbf{$\gamma_{v',v_1}$}};
\draw[line width=1pt] (3.7,-0.4) node {\textbf{$\gamma_{v',v_k}$}};
\draw[->,line width=1pt] (2.1,2.9) -- (1.2,1.65); 		%V->V1
\draw[->,line width=1pt] (2.5,2.9) -- (3.4,1.65); 		%V->Vk
\draw[->,line width=1pt] (1.2,0.35) -- (2.1,-0.9); 		%V1->V'
\draw[->,line width=1pt] (3.4,0.35) -- (2.5,-0.9);		%Vk->V'
\end{scope}
\end{tikzpicture}
\caption{Parallel connection of vertices $v_1,\ldots,v_k$.}
\label{fig:par_large_general} 
\end{figure}
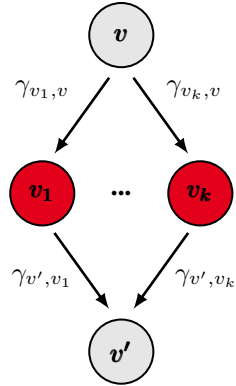

The corresponding gain matrix is given by
\begin{equation}\label{gamma_par}
\Gamma=\left(
\begin{array}{cccccccc}
\ldots&\ldots&\ldots&\ldots&\ldots&\ldots&\ldots&\ldots\\
\ldots&\ldots&0&\ldots&0&0&\ldots&\ldots\\
\ldots&0&0&\ldots&0&\gamma_{v_1,v}&0&\ldots\\
\ldots&\vdots&\vdots&\ddots&\vdots&\vdots&\vdots&\ldots\\
\ldots&0&0&\ldots&0&\gamma_{v_k,v}&0&\ldots\\
\ldots&0&\gamma_{v',v_1}&\ldots&\gamma_{v',v_k}&\dots&\ldots&\ldots\\
\ldots&\ldots&0&\ldots&0&0&\ldots&\ldots\\
\ldots&\ldots&\ldots&\ldots&\ldots&\ldots&\ldots&\ldots
\end{array}
\right).
\end{equation}
The cycle condition  \eqref{cycle} for the cycles that include nodes from $\{v_1,\dots,v_k\}$  looks as follows:
\begin{equation}\label{cycle_large_parallel}
\ldots\circ\gamma_{v',v_i}\circ\gamma_{v_i,v}\circ\ldots<\id.
\end{equation}
\subsubsection{Aggregation of gains}\quad\\
Based on this structure a possibility to attain a graph of a smaller size is to
aggregate the vertices connected in parallel to a single vertex and to leave
the structure of the remaining graph as it is. We denote the new vertex by
$J$. A cut-out of the new reduced graph is shown in Figure~\ref{fig:par_small_general}.

\begin{figure}[tbh]
\centering
\begin{tikzpicture}[scale=0.8] 
\begin{scope}[>=latex]
\filldraw[fill=grau,line width=0.75pt] (2.3,3.6) circle (15pt) node {$\pmb{J}$};
\filldraw[fill=grau,line width=0.75pt] (2.3,0.2) circle (15pt) node {$\pmb{v'}$};
\draw[line width=1pt] (1.6,1.7) node {\textbf{$\widetilde{\gamma}_{v',J}$}};
\draw[->,line width=1pt] (2.3,2.8) -- (2.3,0.85); 		%V->V'
\end{scope}
\end{tikzpicture}
\caption{Aggregation of vertices $v_1,\ldots,v_k,v$.}
\label{fig:par_small_general} 
\end{figure}
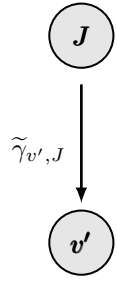

So, we consider the reduced graph $\widetilde{G}=(\widetilde{\V},\widetilde{\E},\widetilde{\Gamma})$, where the vertices are given by
\begin{align}\label{par:tildeV}
 \tilde \V=(\V\setminus (\V_J\cup \{v\})) \cup J
\end{align}
and the edges are given by
\begin{multline}\label{par:tildeE}
\tilde \E= \E \,  \setminus \, (\{(v,w),(w,v') : w \in \V_J\}\cup(v,v'))\,  \cup(J,v')\,\cup\, \{(u,J): \, (u,v) \in \E\}.
\end{multline}

The corresponding weighted adjacency matrix $\tilde \Gamma$ of the dimension $n-k$ can be obtained from $\Gamma$,
where the rows and columns corresponding to the vertices $v,v_1,\ldots,v_k$ are
replaced by a row and column corresponding to the new vertex $J$. The weights are
then given by  
\begin{eqnarray}\label{gain_parallel}
\widetilde{\gamma}_{v',J}:=\max\{\gamma_{v',v_1}\circ\gamma_{v_1,v},\dots,\gamma_{v',v_k}\circ\gamma_{v_k,v},\gamma_{v',v}
\},
\end{eqnarray}
\begin{eqnarray}\label{gain_parallel2}
\widetilde{\gamma}_{J,v'}:=\gamma_{v,v'},\quad \widetilde{\gamma}_{J,j}:=\gamma_{v,j},\quad \widetilde{\gamma}_{j,J}:=\gamma_{j,J},\quad j\in \V\setminus(\V_J\cup\{v,v'\}).
\end{eqnarray}
Other gains stay the same, i.e. 
\begin{eqnarray}\label{gain_parallel3}
\widetilde{\gamma}_{ij}:=\gamma_{ij}, i,j\neq J.
\end{eqnarray}
The small gain condition  \eqref{sgc_max} corresponding to the reduced gain matrix $\tilde \Gamma$ has the following properties.
\begin{theorem}\label{theorem_par}
Consider a gain matrix $\Gamma$ as in \eqref{gamma_par} where the nodes $\{v_1,\ldots,v_k\}$ of the corresponding graph are connected in parallel. 
Then condition \eqref{sgc_max} holds for the matrix $\Gamma$ if and only if condition \eqref{sgc_max} holds for aggregated the matrix $\tilde \Gamma$ with gains defined in \eqref{gain_parallel}-\eqref{gain_parallel3}. 

Furthermore, if there were $p$ cycles that include node $v_i$, then the number of cycles to be checked in the cycle condition \eqref{cycle} corresponding to the reduced matrix $\tilde \Gamma$ is decreased by $p(k-\delta_{v',v})$, where $\delta_{v',v}:=0$, if $\gamma_{v',v}\neq 0$ and $\delta_{v',v}:=1$ otherwise.
\end{theorem}
\begin{proof}
Let condition \eqref{sgc_max} for the gains defined in \eqref{gain_parallel}-\eqref{gain_parallel3}  hold. Then the cycle condition  \eqref{cycle} for these gains holds. In particular, for the cycles containing the gain $\widetilde{\gamma}_{v',J}$ the following inequality holds:
\begin{equation}\label{cycle_small_parallel}
\ldots\circ\widetilde{\gamma}_{v',J}\circ\ldots<\id.
\end{equation} 
From the definition of the gain $\widetilde{\gamma}_{v',J}$ in \eqref{gain_parallel}, condition \eqref{cycle_large_parallel} holds. Condition  \eqref{cycle} on the other cycles is satisfied straightforwardly. Thus $\Gamma$ satisfies \eqref{sgc_max}.

For the proof in the opposite direction we consider the cycle condition \eqref{cycle_large_parallel} and from  \eqref{gain_parallel} obtain \eqref{cycle_small_parallel}, i.e. cycle condition for the aggregated gain $\widetilde{\gamma}_{v',J}$. The rest cycle conditions are satisfied straightforwardly. Thus, the matrix $\tilde\Gamma$ satisfies \eqref{sgc_max}.

If there were $p$ cycles that include node $v_i$ in the large graph, then the number of the cycles that include a node from $\{v_1,\ldots,v_k\}$ is $p\cdot k$. If $\gamma_{v',v}\neq 0$, then the number of cycles with nodes $\{v_1,\ldots,v_k\}$ and gain $\gamma_{v',v}$ is $p\cdot (k+1)$. After the aggregation of the gains these cycles will "coincide", thus the number of the cycles to be checked in the small gain condition \eqref{sgc_max} is decreased by $p(k-\delta_{v',v})$.
\end{proof}
Again, to show that a system of the form \eqref{ws} is ISS, it is enough to verify the small gain condition corresponding to the reduced gain matrix.
\begin{corollary}\label{cor:par}
Consider interconnected system \eqref{ws} and assume that all subsystems in \eqref{is} are ISS with gains as in \eqref{is_iss}. If condition \eqref{sgc_max} holds for the gains defined in \eqref{gain_parallel}-\eqref{gain_parallel3}, then the system \eqref{ws} is ISS.
\end{corollary}
\begin{proof}
The assertion follows from Theorem~\ref{theorem_par} and Theorem~\ref{thm:iss}.
\end{proof}
\subsubsection{Construction of an $\Omega$-path}\quad\\
Again we can calculate an $\Omega$-path for a large gain matrix having an $\Omega$-path corresponding for the reduced one and in opposite direction.
\begin{proposition}\label{proposition_red_par}
Consider a gain matrix $\Gamma$ and the corresponding reduced gain matrix $\tilde \Gamma$ with gains defined in \eqref{gain_parallel}-\eqref{gain_parallel3}. Then:\\
(i) If an $\Omega$-path $\widetilde{\sigma}$  for $\tilde \Gamma$  satisfying \eqref{sigma cond 2'} is given, then an $\Omega$-path $\bar{\sigma}$ for the matrix $\Gamma$ can be constructed as 
\begin{equation}\label{omega_parallel}
\bar{\sigma}_{w}:=
\left\{
\begin{array}{ll}
\gamma_{w,v}\circ\widetilde{\sigma}_J,&\mbox{ if } w\in\{v_1,\dots,v_k\},\\
\widetilde{\sigma}_w,&\mbox{ otherwise }. 
\end{array}
\right.
\end{equation}
(ii)\quad  If an $\Omega$-path $\bar{\sigma}$ for $ \Gamma$  satisfying \eqref{sigma cond 2'} is given, then an $\Omega$-path $\tilde{\sigma}$ for the matrix $\tilde\Gamma$ can be constructed as 
\begin{equation}\label{omega_parallel_2}
\bar{\sigma}_{w}:=
\left\{
\begin{array}{ll}
\bar{\sigma}_w &\mbox{ if } w\in\tilde \V\setminus J,\\
\bar{\sigma}_v,&\mbox{ if } w=J;
\end{array}
\right.
\end{equation}
\end{proposition}
\begin{proof}\quad\\
{\it Proof of (i)}:\\
%Let us check whether $\Omega$-path $\bar{\sigma}$ defined in \eqref{omega_parallel} is an  $\Omega$-path for a large network. 
We assume that an $\Omega$-path $\widetilde{\sigma}$ for the small gain matrix $\widetilde{\sigma}$ is known. In particular, by \eqref{sigma cond 2'}
 $\widetilde{\Gamma}(\widetilde{\sigma})\leq\widetilde{\sigma}$ holds.
Let us check whether an $\Omega$-path $\bar{\sigma}$ defined in \eqref{omega_parallel} is an  $\Omega$-path for the large matrix $\Gamma$. To this end we need to check \eqref{sigma cond 2'}.

For the components $\Gamma(\bar{\sigma})_w$, $w\not\in\{v_1,\dots,v_k,v'\}$ the inequality \eqref{sigma cond 2'} holds straightforwardly.  Consider now  $\Gamma(\bar{\sigma})_w$, $w\in\{v_1,\dots,v_k\}$. Applying \eqref{gain_parallel}-\eqref{gain_parallel3} and \eqref{omega_parallel} we obtain:
\begin{eqnarray*}
\Gamma(\bar{\sigma})_{w}&=&\gamma_{w,v}\circ\bar{\sigma}_v=\widetilde{\sigma}_w;\\
\Gamma(\bar{\sigma})_{v'}&=&\max\{\gamma_{v',1}(\bar{\sigma}_1),\dots,\gamma_{v',v_1}(\bar{\sigma}_{v_1}),\dots,\gamma_{v',v_1}(\bar{\sigma}_{v_k}),\dots,\gamma_{v',n}(\bar{\sigma}_n)\}\\
&=&\max\{\widetilde{\gamma}_{v',1}(\widetilde{\sigma}_{1}),\dots,\underbrace{\gamma_{v',v_1}{\circ}\gamma_{v_1,v}{\circ}\widetilde{\sigma}_J,\dots,
\gamma_{v',v_k}{\circ}\gamma_{v_k,v}{\circ}\widetilde{\sigma}_J}_{\widetilde{\gamma}_{v',J}{\circ}\widetilde{\sigma}_J},\dots,\widetilde{\gamma}_{v',n}(\widetilde{\sigma}_n)\}\\
&=&\max\{\widetilde{\gamma}_{v',1}(\widetilde{\sigma}_{1}),\dots,\widetilde{\gamma}_{v',J}\circ\widetilde{\sigma}_J,\dots,\widetilde{\gamma}_{v',n}(\widetilde{\sigma}_{n})\\
&\leq&\widetilde{\sigma}_{v'}=\bar{\sigma}_{v'}.
\end{eqnarray*}
Thus  $\Gamma(\bar{\sigma})\leq\bar{\sigma}$ and $\bar{\sigma}$ is an $\Omega$-path corresponding to the large gain matrix $\Gamma$.\\
{\it Proof of (ii)}:\\
Assume now that an $\Omega$-path $\bar{\sigma}$ for the large gain matrix $\Gamma$ is known. Let us check whether $\Omega$-path $\tilde{\sigma}$ defined in \eqref{omega_parallel_2} is an  $\Omega$-path for the small gain matrix $\tilde\Gamma$. To this end we need to check \eqref{sigma cond 2'} for $\tilde{\sigma}$.
 
For the components $\tilde\Gamma(\tilde{\sigma})_w$, $w\neq v'$ the inequality  \eqref{sigma cond 2'} holds straightforwardly.  Consider now  $\tilde\Gamma(\tilde{\sigma})_{v'}$. Applying \eqref{gain_parallel}, \eqref{omega_parallel_2} and \eqref{sigma cond 2'} for $w\neq v'$ we obtain:
\begin{eqnarray*}
\tilde\Gamma(\tilde{\sigma})_{v'}&=&\max\{\tilde\gamma_{v',1}\circ\tilde{\sigma}_1,\ldots,\tilde\gamma_{v',J}\circ\tilde{\sigma}_J,\ldots,\tilde\gamma_{v',n}\circ\tilde{\sigma}_n\}\\
&=&\max\{\tilde\gamma_{v',1}\circ\tilde{\sigma}_1,\ldots,\max\{\gamma_{v',v_1}\circ\gamma_{v_1,v},\dots,\gamma_{v',v_k}\circ\gamma_{v_k,v},\gamma_{v',v}\}\circ\bar{\sigma}_v,\ldots,\tilde\gamma_{v',n}\circ\tilde{\sigma}_n\}\\
&\leq&\max\{\tilde\gamma_{v',1}\circ\tilde{\sigma}_1,\ldots,\gamma_{v',v_1}\circ\bar\sigma_{v_1},\dots,\gamma_{v',v_k}\circ\bar\sigma_{v_k},\bar\sigma_{v'},\ldots,\tilde\gamma_{v',n}\circ\tilde{\sigma}_n\}\\
&\leq&\max\{\gamma_{v',1}\circ\bar{\sigma}_1,\ldots,\bar\sigma_{v'},\ldots,\gamma_{v',n}\circ\bar{\sigma}_n\}\\
&\leq&\bar\sigma_{v'}=\tilde\sigma_{v'}
\end{eqnarray*}
Thus  $\tilde\Gamma(\tilde{\sigma})\leq\tilde{\sigma}$ and $\tilde{\sigma}$ is an $\Omega$-path corresponding to the small gain matrix $\tilde\Gamma$.
\end{proof}
\begin{corollary}
Consider a system of the form \eqref{ws} that is an interconnection of the subsystems
\eqref{is}. Assume that each subsystem {\it i} of \eqref{is} has
an ISS Lyapunov function $V_i$ with corresponding ISS-Lyapunov
gains $\gamma_{ij}, \gamma_i, i,j=1,\ldots,n$ as in \eqref{iss_lyap}. if \eqref{sgc_max} holds for $\widetilde\Gamma$ defined by \eqref{gain_parallel}-\eqref{gain_parallel3}, then the system \eqref{ws} is has an ISS-Lyapunov function and an ISS-Lyapunov function is given by \eqref{Lyapunov_function} with $\sigma$ from \eqref{omega_parallel}.
\end{corollary}
\begin{proof}
The assertion follows from Theorem~\ref{thm:iss_lyap} and Proposition~\ref{proposition_red_par}.
\end{proof}
\subsection{Aggregation of almost disconnected subgraphs}
A further structure in the network, that suggests itself to a reduction is
given by subgraphs which are connected to the remainder of the network through
just a single vertex. So, we consider a set of vertices $\V_J=\{v_1,...,v_k\}$
and a distinguished vertex $v^*\in \V\setminus \V_J$ 
such that any path from $v_i, i=1,\ldots,l$ to the remainder of the vertices in
$\V\setminus \V_J$, and any path from $V\setminus \V_J$ to $\V_J$ necessarily
passes through the vertex $v^*$. If we assume that the whole graph is
strongly connected, this implies in particular, that the subgraph induced by $\V_J \cup \{
v^*\}$ is by itself strongly connected.

In Figure~\ref{pic:subgraph_large} an example graph is shown, where the
vertices $\V_J =\{v_1,\ldots,v_k\} $ are connected with the rest of
the graph only through the vertex $v^{\ast}$.
\begin{figure}[tbh]
\centering
\begin{tikzpicture}[scale=0.6] %Grafik 7
\begin{scope}[>=latex]
\filldraw[fill=rot,line width=0.75pt] (0,5) circle (15pt) node {$\pmb{v_1}$};
\filldraw[fill=rot,line width=0.75pt] (2.1,6.65) circle (15pt) node {$\pmb{v_2}$};
\filldraw[fill=rot,line width=0.75pt] (2.5,4) circle (15pt) node {$\pmb{v_3}$};
\draw[dashed] (1.533,5.2167) circle (68pt);
\filldraw[fill=grau,line width=0.75pt] (4.5,2) circle (15pt) node {$\pmb{v^*}$};
\filldraw[fill=grau,line width=0.75pt] (6.75,3) circle (15pt); % or
\filldraw[fill=grau,line width=0.75pt] (4.5,-0.5) circle (15pt); % u
\filldraw[fill=grau,line width=0.75pt] (6.75,-1.5) circle (15pt); % ur
\filldraw[fill=grau,line width=0.75pt] (8.5,0.75) circle (15pt); % r
\draw[->,line width=1pt] (1.875,4.25) -- (0.625,4.75); 			%V3->V1
\draw[->,line width=1pt] (0.525,5.4125) -- (1.575,6.2375); 		%V1->V2
\draw[<-,line width=1pt] (2.3325,6.0075) -- (2.5325,4.6825); 		%V3->V2
\draw[->,line width=1pt] (2.0675,5.9675) -- (2.2675,4.6425); 		%V2->V3
\draw[->,line width=2pt] (4,2.3) -- (2.8,3.5);			%V3->V
\draw[->,line width=2pt] (3,3.7) -- (4.2,2.5);		%V->V3
\draw[->,line width=1pt] (6.1875,2.75) -- (5.0625,2.25);
\draw[->,line width=1pt] (8.0625,1.3125) -- (7.1875,2.4375);
\draw[->,line width=1pt] (5.1,-0.3125) -- (7.9,0.5625);
\draw[->,line width=1pt] (6.75,2.325) -- (6.75,-0.825);
\draw[->,line width=1pt] (4.5,1.375) -- (4.5,0.125);
\draw[->,line width=1pt] (7.1875,-0.9375) -- (8.0625,0.1875);
\draw[->,line width=1pt] (6.1375,-1.3625) -- (5.0125,-0.8625);
\draw[->,line width=1pt] (5.1125,-0.6375) -- (6.2375,-1.1375);
\end{scope}
\end{tikzpicture}
\caption{The subgraph consisting of the vertices $\V_J=\{v_1,v_2,v_3\}$ is almost disconnected from the graph.}
\label{pic:subgraph_large}
\end{figure}
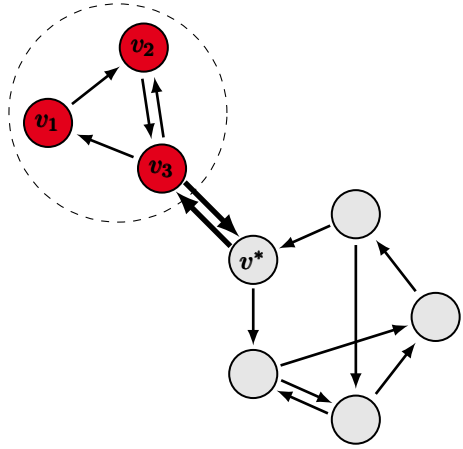

The cycles in \eqref{cycle} that include nodes only from $\{v_1,\dots,v_k,v^*\}$ look as follows:
\begin{equation}\label{cycle_subgraph}
\gamma_{k_1,k_2} \circ \gamma_{k_2,k_3} \circ \dots \circ \gamma_{k_{r-1},k_r} < \id,
\end{equation}
for all $(k_1,...,k_r) \in \{v_1,...,v_k,v^*\}^r$ with $k_1=k_r$.
\subsubsection{Aggregation of gains}\quad\\
To reduce the network size we aggregate the vertices of the subgraph $\V_J$ with vertex $v^*$ 
and do not change the remainder of the graph. We denote the
new vertex by $J$. For the example in Figure~\ref{pic:subgraph_large} the
reduced graph is shown in Figure~\ref{pic:subgraph_small}. So we consider the
reduced graph $\tilde G=(\tilde \V,\tilde \E,\tilde A)$, where the vertices are
given by
\begin{align}
\label{subgraphvertices}
\tilde \V= \left(\V\setminus (\V_J\cup \{v^*\} )\right) \cup J
\end{align}
and the edges are given by
\begin{multline}\label{subgraphedges}
% \nonumber
\tilde \E= \E\,  \setminus \, \{(w_1,w_2),(v^*,w_1),(w_1,v^*): w_1,w_2 \in \V_J\}\, \\ 
\cup\,\{(J,u): \, u\in \tilde \V,(v^*,u)\in \E\}\\
\cup\,\{(u,J): \, u\in \tilde \V,(u,v^*)\in \E\}.
% \nonumber
\end{multline}

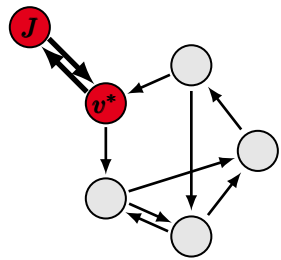
\begin{figure}[tbh]
\centering
\begin{tikzpicture}[scale=0.5] %Grafik 8
\begin{scope}[>=latex]
\filldraw[fill=rot,line width=0.75pt] (4.5,2) circle (15pt) node {$\pmb{v^*}$};
\filldraw[fill=grau,line width=0.75pt] (6.75,3) circle (15pt); % or
\filldraw[fill=grau,line width=0.75pt] (4.5,-0.5) circle (15pt); % u
\filldraw[fill=grau,line width=0.75pt] (6.75,-1.5) circle (15pt); % ur
\filldraw[fill=grau,line width=0.75pt] (8.5,0.75) circle (15pt); % r
\filldraw[fill=rot,line width=0.75pt] (2.5,4) circle (15pt) node {$\pmb{J}$};

\draw[->,line width=2pt] (4,2.3) -- (2.8,3.5);			%V3->V
\draw[->,line width=2pt] (3,3.7) -- (4.2,2.5);		%V->V3
\draw[->,line width=1pt] (6.1875,2.75) -- (5.0625,2.25);
\draw[->,line width=1pt] (8.0625,1.3125) -- (7.1875,2.4375);
\draw[->,line width=1pt] (5.1,-0.3125) -- (7.9,0.5625);
\draw[->,line width=1pt] (6.75,2.325) -- (6.75,-0.825);
\draw[->,line width=1pt] (4.5,1.375) -- (4.5,0.125);
\draw[->,line width=1pt] (7.1875,-0.9375) -- (8.0625,0.1875);
\draw[->,line width=1pt] (6.1375,-1.3625) -- (5.0125,-0.8625);
\draw[->,line width=1pt] (5.1125,-0.6375) -- (6.2375,-1.1375);
\end{scope}
\end{tikzpicture}
\caption{Subgraph $V_J$ and node $v^*$ are merged to vertex $J$.}
\label{pic:subgraph_small}
\end{figure}

The corresponding weighted adjacency matrix $\tilde \Gamma$ of the dimension $n-k+1$ can be obtained from $\Gamma$,
where the rows and columns corresponding to the vertices $v_1,\ldots,v_k$ are
replaced by a row and column corresponding to new vertex $J$. The weights are
then given by 

\begin{eqnarray}\label{gain_subgraph}
\widetilde{\gamma}_{J,v^{\ast}}:=\max_{ (k_1,...,k_r) \in \{v_1,...,v_k,v^*\}^r, k_1=k_r}\{\gamma_{k_1,k_2} \circ \gamma_{k_2,k_3} \circ \dots \circ \gamma_{k_{r-1},k_r}
\},
\end{eqnarray}
\begin{eqnarray}\label{gain_subgraph2}
\widetilde{\gamma}_{v^*,J}=\id.
\end{eqnarray}
Other gains stay the same, i.e. 
\begin{equation}\label{gain_subgraph3}
\widetilde{\gamma}_{ij}:=\gamma_{ij},\, i,j\neq J.
\end{equation}
\begin{theorem}\label{theorem_subgraph} 
Consider a gain matrix $\Gamma$ where the subgraph $\{v_1,\ldots,v_k\}$ of the corresponding graph is strongly connected connected and connected to the remainder of the graph only through one node. Then condition \eqref{sgc_max} holds for the matrix $\Gamma$ if and only if condition \eqref{sgc_max} holds for aggregated the matrix $\tilde \Gamma$ with gains defined in \eqref{gain_subgraph}-\eqref{gain_subgraph3}. 

If there were $p$ cycles that include nodes only from $\V_J\cup\{v^*\}$, then the number of cycles to be checked in the cycle condition \eqref{cycle} corresponding to the reduced matrix $\tilde \Gamma$ is decreased by $p-1$.
\end{theorem}
\begin{proof}
Let condition \eqref{sgc_max} for the gains defined in \eqref{gain_subgraph}-\eqref{gain_subgraph3}  hold. Then the cycle condition \eqref{cycle} for these gains holds. In particular, for the cycles containing $\widetilde{\gamma}_{v^{\ast},J}$, $\widetilde{\gamma}_{J,v^*}$ the following inequality holds:
\begin{equation}\label{cycle_small_subgraph}
\widetilde{\gamma}_{v^*,J}\circ\widetilde{\gamma}_{J,v^*}<\id.
\end{equation} 
From the definition of the gains $\widetilde{\gamma}_{J,v^*}$ and $\widetilde{\gamma}_{v^*,J}$ in \eqref{gain_subgraph} and \eqref{gain_subgraph2}, condition \eqref{sgc_max} for the large matrix $\Gamma$ holds. Conditions on the other cycles in \eqref{cycle} are satisfied straightforwardly. 

For the proof in the opposite direction we consider the cycle condition \eqref{cycle_subgraph} and from  \eqref{gain_subgraph} and \eqref{gain_subgraph2} we obtain \eqref{cycle_small_subgraph}, i.e. cycle condition for the aggregated gains $\widetilde{\gamma}_{v^{\ast},J}$, $\widetilde{\gamma}_{J,v^*}$. The rest cycle conditions are satisfied straightforwardly. Thus, the matrix $\tilde\Gamma$ satisfies \eqref{sgc_max}.

As instead of  $p$ cycles with nodes only from $V_J\cup\{v^*\}$ we consider only one cycle $\widetilde{\gamma}_{v^*,J}\circ\widetilde{\gamma}_{J,v^*}$, the number of cycles corresponding to the small gain matrix $\tilde \Gamma$ is decreased by $p-1$.
\end{proof}
\begin{corollary}\label{cor:subgraph}
Consider interconnected system \eqref{ws} and assume that all subsystems in \eqref{is} are ISS with gains as in \eqref{is_iss}. If condition \eqref{sgc_max} holds for the gains defined in \eqref{gain_subgraph}-\eqref{gain_subgraph3}, then the system \eqref{ws} is ISS.
\end{corollary}
\begin{proof}
The assertion follows from Theorem~\ref{theorem_subgraph} and Theorem~\ref{thm:iss}.
\end{proof}
\subsubsection{Construction of an $\Omega$-path}\quad\\
Again, we can calculate an $\Omega$-path for a large gain matrix having an $\Omega$-path corresponding for a reduced one.
\begin{proposition}\label{proposition_red_subgraph}
Consider a gain matrix $\Gamma$ and the corresponding reduced gain matrix $\tilde \Gamma$ with gains defined in \eqref{gain_subgraph}-\eqref{gain_subgraph3}. \\
(i) \quad If an $\Omega$-path $\widetilde{\sigma}$ for $\tilde \Gamma$ satisfying \eqref{sigma cond 2'} is given, then an $\Omega$-path $\bar{\sigma}$ for the matrix $\Gamma$ can be constructed as:
\begin{equation}\label{omega_subgraph}
\bar{\sigma}_{w}:=
\left\{
\begin{array}{ll}
\hat{\gamma}_{w,J}\circ\widetilde{\sigma}_J,&\mbox{ if } w\in\{v_1,\dots,v_k\},\\
\widetilde{\sigma}_w,&\mbox{ otherwise },
\end{array}
\right.
\end{equation} 
where 
\begin{equation}\label{path_subgraph}
\hat{\gamma}_{w,J}:=\max\limits_{(k_1,\ldots,k_r)\in\{v_1,\ldots,v_k\}^r,k_i\neq k_j}\{\gamma_{w,v_{k_1}}\circ\gamma_{v_{k_1},v_{k_2}}\circ\ldots\circ\gamma_{v_{k_r},J}\},
\end{equation}
i.e. the maximum over compositions of minimal paths (all nodes are different) from node $w$ to node $J$. \\
(ii)\quad  If an $\Omega$-path $\bar{\sigma}$ for $ \Gamma$  satisfying \eqref{sigma cond 2'} is given, then an $\Omega$-path $\tilde{\sigma}$ for the matrix $\tilde\Gamma$ can be constructed as 
\begin{equation}\label{omega_subgraph_2}
\bar{\sigma}_{w}:=
\left\{
\begin{array}{ll}
\bar{\sigma}_w &\mbox{ if } w\in\tilde \V\setminus J,\\
\bar{\sigma}_{v^*},&\mbox{ if } w=J;
\end{array}
\right.
\end{equation}
\end{proposition}
\begin{proof}\quad\\
{\it Proof of (ii)}:\\
We assume that an $\Omega$-path $\widetilde{\sigma}$ for the small gain matrix $\widetilde{\sigma}$ is known. In particular, by \eqref{sigma cond 2'}
 $\widetilde{\Gamma}(\widetilde{\sigma})\leq\widetilde{\sigma}$ holds.
Let us check whether an $\Omega$-path $\bar{\sigma}$ defined in \eqref{omega_subgraph} is an  $\Omega$-path for the large matrix $\Gamma$. To this end we need to check \eqref{sigma cond 2'}.

For the components $\Gamma(\bar{\sigma})_w$, $w\not\in\{v_1,\dots,v_k\}$ the inequality \eqref{sigma cond 2'} holds straightforwardly.  
Consider now  $\Gamma(\bar{\sigma})_w$, $w\in\{v_1,\dots,v_k\}$. From Theorem~\eqref{theorem_subgraph} the cycle condition \eqref{cycle} holds for all $v_1, \ldots, v_k$. Then applying \eqref{cycle}, \eqref{omega_subgraph} and \eqref{path_subgraph} we obtain
\begin{equation}\label{proof_omega_subgraph}
\begin{array}{lll}
\Gamma(\bar{\sigma})_w&=&\max\{\gamma_{w,v_{k_1}}\circ\bar{\sigma}_{k_1},\ldots,\gamma_{w,v_{k_r}}\circ\bar{\sigma}_{k_r}\}\\
&=&\max\{\gamma_{w,v_{k_1}}\circ \hat{\gamma}_{v_{k_1},J}\circ\widetilde{\sigma}_J,\ldots,\gamma_{w,v_{k_r}}\circ \hat{\gamma}_{v_{k_r},J}\circ\widetilde{\sigma}_J\}\\
&=&\hat{\gamma}_{w,J}\circ\widetilde{\sigma}_J\\
&=&\bar{\sigma}_{w}
\end{array}
\end{equation}
Thus \eqref{sigma cond 2'} holds for any $w\in\{v_1,\ldots,v_k\}$ and thus $\bar{\sigma}$ is an $\Omega$-path corresponding to the large gain matrix $\Gamma$.\\
{\it Proof of (ii)}:\\
Assume now that an $\Omega$-path $\bar{\sigma}$ for the large gain matrix $\Gamma$ is known. Let us check whether $\Omega$-path $\tilde{\sigma}$ defined in \eqref{omega_subgraph_2} is an  $\Omega$-path for the small gain matrix $\tilde\Gamma$. To this end we need to check \eqref{sigma cond 2'} for $\tilde{\sigma}$.
 
For the components $\tilde\Gamma(\tilde{\sigma})_w$, $w\not\in\{v^*,J\}$ the inequality  \eqref{sigma cond 2'} holds straightforwardly.  Consider now  $\tilde\Gamma(\tilde{\sigma})_{v^*}$ and $\tilde\Gamma(\tilde{\sigma})_{J}$ . Applying \eqref{cycle}, \eqref{gain_subgraph}, \eqref{omega_subgraph_2} and \eqref{sigma cond 2'} we obtain:
\begin{eqnarray*}
\tilde\Gamma(\tilde{\sigma})_{v^*}&=&\max\{\tilde\gamma_{v^*,1}\circ\tilde\sigma_{1},\ldots,\tilde\gamma_{v^*,J}\circ\tilde\sigma_J,\ldots,\tilde\gamma_{v',n}\circ\tilde{\sigma}_n\}\\
&=&\max\{\tilde\gamma_{v^*,1}\circ\tilde\sigma_{1},\ldots,\tilde\sigma_J,\ldots,\tilde\gamma_{v',n}\circ\tilde{\sigma}_n\}\\
&=&\max\{\tilde\gamma_{v^*,1}\circ\tilde\sigma_{1},\ldots,\tilde{\sigma}_{v^*},\ldots,\tilde\gamma_{v',n}\circ\tilde{\sigma}_n\}\\
&\leq&\tilde{\sigma}_{v^*}
\end{eqnarray*}
and
\begin{eqnarray*}
\tilde\Gamma(\tilde{\sigma})_{J}&=&\tilde\gamma_{J,v^*}\circ\tilde\sigma_{v^*}\\
&=&\max_{ (k_1,...,k_r) \in \{v_1,...,v_k,v^*\}^r, k_1=k_r}\{\gamma_{k_1,k_2} \circ \gamma_{k_2,k_3} \circ \dots \circ \gamma_{k_{r-1},k_r}\}\circ\tilde\sigma_{v^*}\\
&\leq&\tilde\sigma_{v^*}\\
&=&\tilde\sigma_{J}
\end{eqnarray*}
Thus  $\tilde\Gamma(\tilde{\sigma})\leq\tilde{\sigma}$ and $\tilde{\sigma}$ is an $\Omega$-path corresponding to the small gain matrix $\tilde\Gamma$.
\end{proof}
\begin{corollary}
Consider a system of the form \eqref{ws} that is an interconnection of the subsystems
\eqref{is}. Assume that each subsystem {\it i} of \eqref{is} has
an ISS-Lyapunov function $V_i$ with the corresponding ISS-Lyapunov
gains $\gamma_{ij}, \gamma_i, i,j=1,\ldots,n$ as in \eqref{iss_lyap}. If \eqref{sgc_max} holds for $\tilde \sigma$ defined by \eqref{gain_subgraph}-\eqref{gain_subgraph3}, then the system \eqref{ws} has an ISS-Lyapunov function and an ISS-Lyapunov function is given by \eqref{Lyapunov_function} with $\sigma$ from \eqref{omega_subgraph}.
\end{corollary}
\begin{proof}
The assertion follows from Theorem~\ref{thm:iss_lyap} and Proposition~\ref{proposition_red_subgraph}.
\end{proof}

\section{Application of aggregation rules}\label{sec:Example} 

The properties of aggregation rules can be summarized in the following corollary.
\begin{corollary}\label{thm:main}
Consider interconnected system \eqref{ws} and assume that all subsystems in \eqref{is} are ISS with gains as in \eqref{is_iss}. Let gain matrix $\tilde\Gamma$ is obtained by step-by-step application of aggregation rules for parallel, sequential and almost disconnected subgraph in any order. 
If condition \eqref{sgc_max} holds for matrix $\tilde\Gamma$, then the system \eqref{ws} is ISS.
\end{corollary}
\begin{proof}
The proof follows from Corollary~\ref{cor:seq}, Corollary~\ref{cor:par},  and Corollary~\ref{cor:subgraph}.
\end{proof}

Let us apply the obtained reduction rules on the following example with the network of 30 nodes, see Figure~\ref{figure:example_large}.  We assume that all the subsystems are ISS with the following gains $\gamma_{ij}$: $\gamma_{3,1}(t)=\frac{5}{6}t^2$, $\gamma_{1,2}(t)=2t$, $\gamma_{2,13}(t)=\frac{1}{5}\sqrt{t}$, $\gamma_{13,30}(t)=t$, $\gamma_{30,29}(t)=4t^2$, $\gamma_{29,28}(t)=\sqrt{t}$, $\gamma_{29,23}(t)=\frac{3}{10}t^2$, $\gamma_{23,17}(t)=\sqrt{t}$, $\gamma_{17,12}(t)=3t$, $\gamma_{12,7}(t)=t^2$, $\gamma_{7,4}(t)=\sqrt{t}$, $\gamma_{3,4}(t)=\sqrt{t}$, $\gamma_{4,3}(t)=\frac{4}{5}t^2$, $\gamma_{4,6}(t)=t$, $\gamma_{6,3}(t)=\sqrt{t}$, $\gamma_{15,11}(t)=t^2$, $\gamma_{11,16}(t)=\frac{2}{3}t$, $\gamma_{16,11}(t)=t$, $\gamma_{16,15}(t)=\sqrt{t}$, $\gamma_{15,16}(t)=t^2$, $\gamma_{19,15}(t)=\sqrt{t}$, $\gamma_{20,15}(t)=2\sqrt{t}$, $\gamma_{26,19,}(t)=\frac{1}{3}t$, $\gamma_{26,20}(t)=\frac{1}{4}t$, $\gamma_{22,26}(t)=\frac{1}{9}t^2$, $\gamma_{21,22}(t)=2\sqrt{t}$, $\gamma_{22,21}(t)=\frac{1}{4}t^2$, $\gamma_{16,22}(t)=3\sqrt{t}$, 
$\gamma_{5,3}(t)=\frac{1}{2}t ,\gamma_{14,8}(t)=\frac{1}{8}t^2, \gamma_{8,5}(t)=\sqrt{t}, \gamma_{14,9}(t)=\frac{13}{16}t^2, \gamma_{9,5}(t)=2\sqrt{t}, \gamma_{14,10}(t)=t^2, \gamma_{10,5}(t)=\frac{1}{2}\sqrt{t}, \gamma_{28,27}(t)=\frac{1}{3}t^2, \gamma_{27,25}(t)=\sqrt{t}, \gamma_{28,24}(t)=\frac{1}{2}t, \gamma_{24,18}(t)=t^2, \gamma_{25,18}=t^2, \gamma_{18,14}(t)=\sqrt{t}$. 

To establish ISS of this network we can apply Theorem~\ref{thm:iss}. To this end we need to verify small gain condition \eqref{sgc_max} or equivalent cycle condition \eqref{cycle}.  
The network in Figure~\ref{figure:example_large} has 29 minimal cycles, i.e. cycles with $k_i\neq k_j$ for $i\neq j$ other than for $k_0=k_j$. Note that already at this stage it is rather difficult to identify all the minimal cycles. To this end we can use a numerical algorithm, see for example \cite{Jon75} and \cite{Tar72}. However, finally we will need any way to verify the cycle conditions analytically. To apply Theorem~\ref{thm:iss} we need to check 29 cycle conditions. % ({\bf will be deleted later!}):

\begin{figure*}[t]
\begin{minipage}{0.47\linewidth}
\centering
\begin{tikzpicture}[scale=0.25] %Grafik 1
\begin{scope}[>=latex]
%\draw [help lines,xscale=3,yscale=3] (0,0) grid (8,8);
\filldraw[fill=grau,line width=0.25pt] (12,24) circle (20pt) node (X1) {\tiny $\pmb{1}$};
%\node (X1) at (12,24) [circle=1pt,fill=grau,draw,line width=0.25pt] {\footnotesize$\pmb{\phantom{\ }1\phantom{\ }}$};

\filldraw[fill=grau,line width=0.25pt] (21,24) circle (20pt) node (X2) {\tiny$\pmb{2}$};
%\node (X2) at (21,24) [circle=1pt,fill=grau,draw,line width=0.25pt] {\tiny$\pmb{\phantom{\ }2\phantom{\ }}$};

\filldraw[fill=grau,line width=0.25pt] (8,21) circle (20pt) node (X3){\tiny$\pmb{3}$};
%\node (X3) at (8,21) [circle=1pt,fill=grau,draw,line width=0.25pt] {\tiny$\pmb{\phantom{\ }3\phantom{\ }}$};

\filldraw[fill=grau,line width=0.25pt] (18,21) circle (20pt) node (X4){\tiny$\pmb{4}$};
%\node (X4) at (18,21) [circle=1pt,fill=grau,draw,line width=0.25pt] {\tiny$\pmb{\phantom{\ }4\phantom{\ }}$};

\filldraw[fill=grau,line width=0.25pt] (3,18) circle (20pt) node (X5){\tiny$\pmb{5}$};
%\node (X5) at (3,18) [circle=1pt,fill=grau,draw,line width=0.25pt] {\tiny$\pmb{\phantom{\ }5\phantom{\ }}$};

\filldraw[fill=grau,line width=0.25pt] (12,18) circle (20pt) node (X6){\tiny$\pmb{6}$};
%\node (X6) at (12,18) [circle=1pt,fill=grau,draw,line width=0.25pt] {\tiny$\pmb{\phantom{\ }6\phantom{\ }}$};

\filldraw[fill=grau,line width=0.25pt] (21,18) circle (20pt) node (X7){\tiny$\pmb{7}$};
%\node (X7) at (21,18) [circle=1pt,fill=grau,draw,line width=0.25pt] {\tiny$\pmb{\phantom{\ }7\phantom{\ }}$};

\filldraw[fill=grau,line width=0.25pt] (0.5,15) circle (20pt) node (X8){\tiny$\pmb{8}$};
%\node (X8) at (0.5,15) [circle=1pt,fill=grau,draw,line width=0.25pt] {\tiny$\pmb{\phantom{\ }8\phantom{\ }}$};

\filldraw[fill=grau,line width=0.25pt] (3,15) circle (20pt) node (X9){\tiny$\pmb{9}$};
%\node (X9) at (3,15) [circle=1pt,fill=grau,draw,line width=0.25pt] {\tiny$\pmb{\phantom{\ }9\phantom{\ }}$};

\filldraw[fill=grau,line width=0.25pt] (5.5,15) circle (20pt) node (X10){\tiny$\pmb{10}$};
%\node (X10) at (5.5,15) [circle=1pt,fill=grau,draw,line width=0.25pt] {\tiny$\pmb{10}$};

\filldraw[fill=grau,line width=0.25pt] (12,15) circle (20pt) node (X11){\tiny$\pmb{11}$};
%\node (X11) at (12,15) [circle=1pt,fill=grau,draw,line width=0.25pt] {\tiny$\pmb{11}$};

\filldraw[fill=grau,line width=0.25pt] (21,15) circle (20pt) node (X12){\tiny$\pmb{12}$};
%\node (X12) at (21,15) [circle=1pt,fill=grau,draw,line width=0.25pt] {\tiny$\pmb{12}$};

\filldraw[fill=grau,line width=0.25pt] (25,12) circle (20pt) node (X13){\tiny$\pmb{13}$};
%\node (X13) at (25,12) [circle=1pt,fill=grau,draw,line width=0.25pt] {\tiny$\pmb{13}$};

\filldraw[fill=grau,line width=0.25pt] (3,12) circle (20pt) node (X14){\tiny$\pmb{14}$};
%\node (X14) at (3,12) [circle=1pt,fill=grau,draw,line width=0.25pt] {\tiny$\pmb{14}$};

\filldraw[fill=grau,line width=0.25pt] (12,12) circle (20pt) node (X15){\tiny$\pmb{15}$};
%\node (X15) at (12,12) [circle=1pt,fill=grau,draw,line width=0.25pt] {\tiny$\pmb{15}$};

\filldraw[fill=grau,line width=0.25pt] (18,12) circle (20pt) node (X16){\tiny$\pmb{16}$};
%\node (X16) at (18,12) [circle=1pt,fill=grau,draw,line width=0.25pt] {\tiny$\pmb{16}$};

\filldraw[fill=grau,line width=0.25pt] (21,12) circle (20pt) node (X17){\tiny$\pmb{17}$};
%\node (X17) at (21,12) [circle=1pt,fill=grau,draw,line width=0.25pt] {\tiny$\pmb{17}$};

\filldraw[fill=grau,line width=0.25pt] (3,9) circle (20pt) node (X18){\tiny$\pmb{18}$};
%\node (X18) at (3,9) [circle=1pt,fill=grau,draw,line width=0.25pt] {\tiny$\pmb{18}$};

\filldraw[fill=grau,line width=0.25pt] (8,9) circle (20pt) node (X19){\tiny$\pmb{19}$};
%\node (X19) at (8,9) [circle=1pt,fill=grau,draw,line width=0.25pt] {\tiny$\pmb{19}$};

\filldraw[fill=grau,line width=0.25pt] (12,9) circle (20pt) node (X20){\tiny$\pmb{20}$};
%\node (X20) at (12,9) [circle=1pt,fill=grau,draw,line width=0.25pt] {\tiny$\pmb{20}$};

\filldraw[fill=grau,line width=0.25pt] (15,9) circle (20pt) node (X21){\tiny$\pmb{21}$};
%\node (X21) at (15,9) [circle=1pt,fill=grau,draw,line width=0.25pt] {\tiny$\pmb{21}$};

\filldraw[fill=grau,line width=0.25pt] (18,9) circle (20pt) node (X22){\tiny$\pmb{22}$};
%\node (X22) at (18,9) [circle=1pt,fill=grau,draw,line width=0.25pt] {\tiny$\pmb{22}$};

\filldraw[fill=grau,line width=0.25pt] (21,9) circle (20pt) node (X23){\tiny$\pmb{23}$};
%\node (X23) at (21,9) [circle=1pt,fill=grau,draw,line width=0.25pt] {\tiny$\pmb{23}$};

\filldraw[fill=grau,line width=0.25pt] (0.5,6) circle (20pt) node (X24){\tiny$\pmb{24}$};
%\node (X24) at (0.5,6) [circle=1pt,fill=grau,draw,line width=0.25pt] {\tiny$\pmb{24}$};

\filldraw[fill=grau,line width=0.25pt] (5.5,6) circle (20pt) node (X25){\tiny$\pmb{25}$};
%\node (X25) at (5.5,6) [circle=1pt,fill=grau,draw,line width=0.25pt] {\tiny$\pmb{25}$};

\filldraw[fill=grau,line width=0.25pt] (12,6) circle (20pt) node (X26){\tiny$\pmb{26}$};
%\node (X26) at (12,6) [circle=1pt,fill=grau,draw,line width=0.25pt] {\tiny$\pmb{26}$};

\filldraw[fill=grau,line width=0.25pt] (5.5,3) circle (20pt) node (X27){\tiny$\pmb{27}$};
%\node (X27) at (5.5,3) [circle=1pt,fill=grau,draw,line width=0.25pt] {\tiny$\pmb{27}$};

\filldraw[fill=grau,line width=0.25pt] (3,0) circle (20pt) node (X28){\tiny$\pmb{28}$};
%\node (X28) at (3,0) [circle=1pt,fill=grau,draw,line width=0.25pt] {\tiny$\pmb{28}$};

\filldraw[fill=grau,line width=0.25pt] (12,0) circle (20pt) node (X29){\tiny$\pmb{29}$};
%\node (X29) at (12,0) [circle=1pt,fill=grau,draw,line width=0.25pt] {\tiny$\pmb{29}$};

\filldraw[fill=grau,line width=0.25pt] (21,0) circle (20pt) node (X30){\tiny$\pmb{30}$};
%\node (X30) at (21,0) [circle=1pt,fill=grau,draw,line width=0.25pt] {\tiny$\pmb{30}$};
%
\draw[->,line width=0.25pt] (X1) -- (X3);
\draw[->,line width=0.25pt] (X1) -- (X4);
\draw[->,line width=0.25pt] (X2) -- (X1);
\draw[->,line width=0.25pt] (X13) -- (X2);
\draw[->,line width=0.25pt] (X3) -- (X5);
\draw[->,line width=0.25pt] (X3) -- (X6);
\draw[->,line width=0.25pt] (X6) -- (X4);
\draw[->,line width=0.25pt] (X4) -- (X7);
\draw[->,line width=0.25pt] (X7) -- (X12);
\draw[->,line width=0.25pt] (X12) -- (X17);
\draw[->,line width=0.25pt] (X17) -- (X23);
\draw[->,line width=0.25pt] (X23) -- (X29);
\draw[->,line width=0.25pt] (12.15,17.2) -- (12.15,15.8);
\draw[->,line width=0.25pt] (11.85,15.8) -- (11.85,17.2);
\draw[->,line width=0.25pt] (X11) -- (X15);
\draw[->,line width=0.25pt] (12.8,14.85) -- (17.4,12.55);
\draw[->,line width=0.25pt] (17.25,12.25) -- (12.7,14.55);
\draw[->,line width=0.25pt] (12.8,12.15) -- (17.2,12.15);
\draw[->,line width=0.25pt] (17.2,11.85) -- (12.8,11.85);
\draw[->,line width=0.25pt] (X15) -- (X19);
\draw[->,line width=0.25pt] (X15) -- (X20);
\draw[->,line width=0.25pt] (X20) -- (X26);
\draw[->,line width=0.25pt] (X19) -- (X26);
\draw[->,line width=0.25pt] (X26) -- (X22);
\draw[->,line width=0.25pt] (X22) -- (X16);
\draw[->,line width=0.25pt] (15.8,9.15) -- (17.2,9.15);
\draw[->,line width=0.25pt] (17.2,8.85) -- (15.8,8.85);
\draw[->,line width=0.25pt] (8.8,21.15) -- (17.2,21.15);
\draw[->,line width=0.25pt] (17.2,20.85) -- (8.8,20.85);
\draw[->,line width=0.25pt] (X5) -- (X8);
\draw[->,line width=0.25pt] (X5) -- (X10);
\draw[->,line width=0.25pt] (X5) -- (X9) ;
\draw[->,line width=0.25pt] (X8) -- (X14);
\draw[->,line width=0.25pt] (X9) -- (X14);
\draw[->,line width=0.25pt] (X10) -- (X14);
\draw[->,line width=0.25pt] (X14) -- (X18);
\draw[->,line width=0.25pt] (X18) -- (X24);
\draw[->,line width=0.25pt] (X18) -- (X25);
\draw[->,line width=0.25pt] (X25) -- (X27);
\draw[->,line width=0.25pt] (X27) -- (X28);
\draw[->,line width=0.25pt] (X24) -- (X28);
\draw[->,line width=0.25pt] (X28) -- (X29);
\draw[->,line width=0.25pt] (X29) -- (X30);
\draw[->,line width=0.25pt] (X30) -- (X13);
\draw[->,line width=0.25pt] (X28) to[out=170,in=155] (X3);
\end{scope}
\end{tikzpicture}
\caption{Network of 30 nodes}
\label{figure:example_large}
\end{minipage}
\begin{minipage}{0.47\linewidth}
\centering
\begin{tikzpicture}[scale=0.25] %Grafik 1
\begin{scope}[>=latex]
%\draw [help lines,xscale=3,yscale=3] (0,0) grid (8,8);
\filldraw[fill=grau,line width=0.25pt] (8,21) circle (20pt) node (X3){\tiny$\pmb{3}$};
\filldraw[fill=grau,line width=0.25pt] (3,18) circle (20pt) node (X5){\tiny$\pmb{5}$};
\filldraw[fill=grau,line width=0.25pt] (0.5,15) circle (20pt) node (X8){\tiny$\pmb{8}$};
\filldraw[fill=grau,line width=0.25pt] (3,15) circle (20pt) node (X9){\tiny$\pmb{9}$};
\filldraw[fill=grau,line width=0.25pt] (5.5,15) circle (20pt) node (X10){\tiny$\pmb{10}$};
\filldraw[fill=grau,line width=0.25pt] (3,12) circle (20pt) node (X14){\tiny$\pmb{14}$};
\filldraw[fill=grau,line width=0.25pt] (3,9) circle (20pt) node (X18){\tiny$\pmb{18}$};
\filldraw[fill=grau,line width=0.25pt] (0.5,6) circle (20pt) node (X24){\tiny$\pmb{24}$};
\filldraw[fill=grau,line width=0.25pt] (5.5,6) circle (20pt) node (X25){\tiny$\pmb{25}$};
\filldraw[fill=grau,line width=0.25pt] (5.5,3) circle (20pt) node (X27){\tiny$\pmb{27}$};
\filldraw[fill=grau,line width=0.25pt] (3,0) circle (20pt) node (X28){\tiny$\pmb{28}$};
%Edges
\draw[->,line width=0.25pt] (X3) -- (X5)node[pos=0.5,right,color=black]{\tiny$\frac{1}{2}t$};
\draw[->,line width=0.25pt] (X5) -- (X8) node[pos=0.5,left,color=black]{\tiny$\sqrt{t}$};
\draw[->,line width=0.25pt] (X5) -- (X10)node[pos=-0.05,right,color=black]{\tiny$\frac{1}{2}\sqrt{t}$};
\draw[->,line width=0.25pt] (X5) -- (X9) node[pos=0.75,right,color=black]{\tiny$2\sqrt{t}$};
\draw[->,line width=0.25pt] (X8) -- (X14)node[pos=0.4,left,color=black]{\tiny$\frac{1}{8}t^2$};
\draw[->,line width=0.25pt] (X9) -- (X14)node[pos=0.9,left,color=black]{\tiny$\frac{13}{16}t^2$};
\draw[->,line width=0.25pt] (X10) -- (X14)node[pos=0.5,right,color=black]{\tiny$t^2$};
\draw[->,line width=0.25pt] (X14) -- (X18)node[pos=0.5,right,color=black]{\tiny$\sqrt{t}$};
\draw[->,line width=0.25pt] (X18) -- (X24)node[pos=0.25,left,color=black]{\tiny$t^2$};
\draw[->,line width=0.25pt] (X18) -- (X25)node[pos=0.25,right,color=black]{\tiny$t^2$};
\draw[->,line width=0.25pt] (X25) -- (X27)node[pos=0.5,right,color=black]{\tiny$\sqrt{t}$};
\draw[->,line width=0.25pt] (X27) -- (X28)node[pos=0.5,right,color=black]{\tiny$\frac{1}{3}t^2$};
\draw[->,line width=0.25pt] (X24) -- (X28)node[pos=0.5,left,color=black]{\tiny$\frac{1}{2}t$};
%\draw (-5,24) node {\parbox{4cm}{ \alert{maximal cycle\\ length = 13 nodes}}};
\end{scope}
\end{tikzpicture}
\caption{Subgraph.}
\label{example_large_subraph}
\end{minipage}
\vspace{-0.5cm}
\end{figure*}

The longest minimal cycle consists of 14 nodes.

Consider the sub-graph with nodes 3, 5, 8, 9, 10, 14, 18, 24, 25, 27 and 28, see Figure~\ref{example_large_subraph}.

This subgraph can be aggregated in several steps:
\begin{enumerate}

\item Aggregation of nodes connected in parallel:\\
$\tilde\gamma_{14,5}=\max\{\gamma_{14,8}\circ\gamma_{8,5},\gamma_{14,9}\circ\gamma_{9,5},\gamma_{14,10}\circ\gamma_{10,5}\}$

$=\max\{\frac{1}{2}(\sqrt(t))^2,\frac{13}{16}(2\sqrt(t))^2,(\frac{1}{2}\sqrt(t))^2\}=\frac{13}{4}t$

\item Aggregation of sequentially connected nodes:\\
$\tilde\gamma_{28,25}=\gamma_{28,27}\circ\gamma_{27,25}=\frac{1}{3}(\sqrt(t))^2=\frac{1}{3}t$

\item Aggregation of nodes connected in parallel:\\
$\tilde\gamma_{28,18}
=\max\{\gamma_{28,24}\circ\gamma_{24,18},\gamma_{28,25}\circ\gamma_{25,18}\}=\max\{\frac{1}{2}t^2,\frac{1}{3}t^2\}=\frac{1}{2}t^2$

\item Aggregation of sequentially connected nodes:\\
$\tilde\gamma_{28,3}
=\gamma_{28,18}\circ\gamma_{18,14}\circ\gamma_{14,5}\circ\gamma_{5,3}
=\frac{1}{2}(\sqrt{\frac{13}{4}\frac{1}{2}t})^2=\frac{13}{16}t$.

\end{enumerate}
Thus we obtain the graph with 21 nodes, see Figure~\ref{Figure:21 nodes}.

Let us apply aggregation rules to the rest of the graph in the following order: 
\begin{enumerate}
\item aggregation of almost disconnected subgraph with nodes 11, 15-16, 19-22 and 26;
\item aggregation of sequentially connected nodes 7, 12, 17, 23;
\item  aggregation of sequentially connected nodes 2, 13, 30
\end{enumerate}
Thus we obtain the graph with 7 nodes, see Figure~\ref{Figure:7 nodes}. The gains of the aggregated network are as follows:

$\tilde\gamma_{28,3}(t):=\frac{13}{16}t$ $\tilde\gamma_{3,1}(t):=\frac{5}{6}t^2$, $\tilde\gamma_{4,1}(t):=t$, $\tilde\gamma_{3,4}(t):=\frac{4}{5}t^2$, $\tilde\gamma_{4,3}(t):=\sqrt{t}$, $\tilde\gamma_{6,3}(t):=\sqrt{t}$, $\tilde\gamma_{4,6}(t):=t$, $\tilde\gamma_{6,11}(t):=t$, $\tilde\gamma_{11,6}(t):=\frac{2}{3}t$, $\tilde\gamma_{3,28}(t):=\frac{2}{3}t$, $\tilde\gamma_{29,28}(t):=\sqrt{t}$, $\tilde\gamma_{1,29}(t):=\frac{4}{5}t$, $\tilde\gamma_{29,4}(t):=\frac{9}{10}t$.

\begin{figure*}[t]
\begin{minipage}{0.4\linewidth}
\centering
\begin{tikzpicture}[scale=0.25] %Grafik 1
\begin{scope}[>=latex]
%\draw [help lines,xscale=3,yscale=3] (0,0) grid (8,8);
\filldraw[fill=grau,line width=0.25pt] (12,24) circle (20pt) node (X1) {\tiny $\pmb{1}$};
\filldraw[fill=grau,line width=0.25pt] (21,24) circle (20pt) node (X2) {\tiny$\pmb{2}$};
\filldraw[fill=grau,line width=0.25pt] (8,21) circle (20pt) node (X3){\tiny$\pmb{3}$};
\filldraw[fill=grau,line width=0.25pt] (18,21) circle (20pt) node (X4){\tiny$\pmb{4}$};
%\filldraw[fill=rot,line width=0.25pt] (3,18) circle (20pt) node (X5){\tiny$\pmb{5}$};
\filldraw[fill=grau,line width=0.25pt] (12,18) circle (20pt) node (X6){\tiny$\pmb{6}$};
\filldraw[fill=grau,line width=0.25pt] (21,18) circle (20pt) node (X7){\tiny$\pmb{7}$};
%\filldraw[fill=rot,line width=0.25pt] (0.5,15) circle (20pt) node (X8){\tiny$\pmb{8}$};
%\filldraw[fill=rot,line width=0.25pt] (3,15) circle (20pt) node (X9){\tiny$\pmb{9}$};
%\filldraw[fill=rot,line width=0.25pt] (5.5,15) circle (20pt) node (X10){\tiny$\pmb{10}$};
\filldraw[fill=grau,line width=0.25pt] (12,15) circle (20pt) node (X11){\tiny$\pmb{11}$};
\filldraw[fill=grau,line width=0.25pt] (21,15) circle (20pt) node (X12){\tiny$\pmb{12}$};
\filldraw[fill=grau,line width=0.25pt] (25,12) circle (20pt) node (X13){\tiny$\pmb{13}$};
%\filldraw[fill=rot,line width=0.25pt] (3,12) circle (20pt) node (X14){\tiny$\pmb{14}$};
\filldraw[fill=grau,line width=0.25pt] (12,12) circle (20pt) node (X15){\tiny$\pmb{15}$};
\filldraw[fill=grau,line width=0.25pt] (18,12) circle (20pt) node (X16){\tiny$\pmb{16}$};
\filldraw[fill=grau,line width=0.25pt] (21,12) circle (20pt) node (X17){\tiny$\pmb{17}$};
%\filldraw[fill=rot,line width=0.25pt] (3,9) circle (20pt) node (X18){\tiny$\pmb{18}$};
\filldraw[fill=grau,line width=0.25pt] (8,9) circle (20pt) node (X19){\tiny$\pmb{19}$};
\filldraw[fill=grau,line width=0.25pt] (12,9) circle (20pt) node (X20){\tiny$\pmb{20}$};
\filldraw[fill=grau,line width=0.25pt] (15,9) circle (20pt) node (X21){\tiny$\pmb{21}$};
\filldraw[fill=grau,line width=0.25pt] (18,9) circle (20pt) node (X22){\tiny$\pmb{22}$};
\filldraw[fill=grau,line width=0.25pt] (21,9) circle (20pt) node (X23){\tiny$\pmb{23}$};
%\filldraw[fill=rot,line width=0.25pt] (0.5,6) circle (20pt) node (X24){\tiny$\pmb{24}$};
%\filldraw[fill=rot,line width=0.25pt] (5.5,6) circle (20pt) node (X25){\tiny$\pmb{25}$};
\filldraw[fill=grau,line width=0.25pt] (12,6) circle (20pt) node (X26){\tiny$\pmb{26}$};
%\filldraw[fill=rot,line width=0.25pt] (5.5,3) circle (20pt) node (X27){\tiny$\pmb{27}$};
\filldraw[fill=grau,line width=0.25pt] (3,0) circle (20pt) node (X28){\tiny$\pmb{28}$};
\filldraw[fill=grau,line width=0.25pt] (12,0) circle (20pt) node (X29){\tiny$\pmb{29}$};
\filldraw[fill=grau,line width=0.25pt] (21,0) circle (20pt) node (X30){\tiny$\pmb{30}$};
%Edges
\draw[->,line width=0.25pt] (X1) -- (X3);
\draw[->,line width=0.25pt] (X1) -- (X4);
\draw[->,line width=0.25pt] (X2) -- (X1);
\draw[->,line width=0.25pt] (X13) -- (X2);
%\draw[->,line width=0.25pt] (X3) -- (X5);\draw[->,line width=0.25pt] (X3) -- (X28)node[pos=0.5,left,color=black]{\tiny$\frac{3}{16}t$};
\draw[->,line width=0.25pt] (X3) -- (X6);
\draw[->,line width=0.25pt] (X6) -- (X4);
\draw[->,line width=0.25pt] (X4) -- (X7);
\draw[->,line width=0.25pt] (X7) -- (X12);
\draw[->,line width=0.25pt] (X12) -- (X17);
\draw[->,line width=0.25pt] (X17) -- (X23);\draw[->,line width=0.25pt] (X23) -- (X29);\draw[->,line width=0.25pt] (12.15,17.2) -- (12.15,15.8);
\draw[->,line width=0.25pt] (11.85,15.8) -- (11.85,17.2);
\draw[->,line width=0.25pt] (X11) -- (X15);\draw[->,line width=0.25pt] (12.8,14.85) -- (17.4,12.55);
\draw[->,line width=0.25pt] (17.25,12.25) -- (12.7,14.55);
\draw[->,line width=0.25pt] (12.8,12.15) -- (17.2,12.15);
\draw[->,line width=0.25pt] (17.2,11.85) -- (12.8,11.85);
\draw[->,line width=0.25pt] (X15) -- (X19);
\draw[->,line width=0.25pt] (X15) -- (X20);
\draw[->,line width=0.25pt] (X20) -- (X26);
\draw[->,line width=0.25pt] (X19) -- (X26);
\draw[->,line width=0.25pt] (X26) -- (X22);\draw[->,line width=0.25pt] (X22) -- (X16);\draw[->,line width=0.25pt] (15.8,9.15) -- (17.2,9.15);
\draw[->,line width=0.25pt] (17.2,8.85) -- (15.8,8.85);
\draw[->,line width=0.25pt] (8.8,21.15) -- (17.2,21.15);
\draw[->,line width=0.25pt] (17.2,20.85) -- (8.8,20.85);
\draw[->,line width=0.25pt] (X28) -- (X29);
\draw[->,line width=0.25pt] (X29) -- (X30);
\draw[->,line width=0.25pt] (X30) -- (X13);\draw[->,line width=0.25pt] (X28) to[out=170,in=155] (X3);
%\draw (-5,24) node {\parbox{4cm}{ \alert{maximal cycle\\ length = 13 nodes}}};
\end{scope}
\end{tikzpicture}
\caption{Graph with 21 nodes.}
\label{Figure:21 nodes}
\end{minipage}
\begin{minipage}{0.2\linewidth}
\hspace{0.2cm}
\end{minipage}
\begin{minipage}{0.4\linewidth}
\centering
\begin{tikzpicture}[scale=0.35] %Grafik 2
\begin{scope}[>=latex]
\filldraw[fill=grau,line width=0.25pt] (12,24) circle (20pt) node (X1) {\tiny $\pmb{1}$};
\filldraw[fill=grau,line width=0.25pt] (8,21) circle (20pt) node (X2){\tiny $\pmb{3}$};
\filldraw[fill=grau,line width=0.25pt] (16,21) circle (20pt) node (X3){\tiny $\pmb{4}$};
\filldraw[fill=grau,line width=0.25pt] (12,18) circle (20pt) node (X4){\tiny $\pmb{6}$};
\filldraw[fill=grau,line width=0.25pt] (12,15) circle (20pt) node (X5){\tiny$\pmb{11}$};
\filldraw[fill=grau,line width=0.25pt] (8,12) circle (20pt) node (X6){\tiny $\pmb{28}$};
\filldraw[fill=grau,line width=0.25pt] (16,12) circle (20pt) node (X7){\tiny $\pmb{29}$};
\draw[->,line width=0.25pt] (X1) -- (X2)node[pos=0.25,left]{\tiny$\frac{5}{6}t^2$};
\draw[->,line width=0.25pt] (X1) -- (X3)node[pos=0.75,right]{\tiny$t$};
\draw[->,line width=0.25pt] (8.8,21.15) -- (15.2,21.15)node[pos=0.5,above]{\tiny$\sqrt{t}$};
\draw[->,line width=0.25pt] (15.2,20.85) -- (8.8,20.85)node[pos=0.5,below]{\tiny$\frac{4}{5}t^2$};
\draw[->,line width=0.25pt] (X2) -- (X4)node[pos=0.5,left]{\tiny$\sqrt{t}$};
\draw[->,line width=0.25pt] (X4) -- (X3)node[pos=0.25,right]{\tiny$t$};
\draw[->,line width=0.25pt] (12.15,17.2) -- (12.15,15.8)node[pos=0.5,right]{\tiny$\frac{2}{3}t$};
\draw[->,line width=0.25pt] (11.85,15.8) -- (11.85,17.2)node[pos=0.5,left]{\tiny$t$};
\draw[->,line width=0.25pt] (X6) -- (X7)node[pos=0.5,below]{\tiny$\sqrt{t}$};
\draw[->,line width=0.25pt] (8.15,20.2) -- (8.15,12.8)node[pos=0.85,right]{\tiny$\frac{13}{16}t$};
\draw[->,line width=0.25pt] (7.85,12.8) -- (7.85,20.2)node[pos=0.5,left]{\tiny$\frac{2}{3}t$};
\draw[->,line width=0.25pt] (X7) to[out=5,in=25] (X1)node[pos=0.5,right]{\tiny$\frac{4}{5}t$};
\draw[->,line width=0.25pt] (X3) -- (X7)node[pos=0.5,right]{\tiny$\frac{9}{10}t$};
\end{scope}
\end{tikzpicture}
\caption{Graph with 7 nodes.}
\label{Figure:7 nodes}
\end{minipage}
\vspace{-0.5cm}
\end{figure*}
The reduced graph has now 8 minimal cycles, where the longest maximal cycle has 5 nodes.

To establish ISS of the large network we can apply Corollary~\ref{thm:main} . Thus we need to verify $\tilde\Gamma(s)\not\geq s$ or equivalently the cycle condition \eqref{cycle} corresponding to reduced matrix $\tilde\Gamma$. 

The corresponding cycle condition looks as follows:

$\widetilde\gamma_{3,4}\circ\widetilde\gamma_{4,3}<\id$, 
$\widetilde\gamma_{3,4}\circ\widetilde\gamma_{4,6}\circ\widetilde\gamma_{6,3}<\id$, 
$\widetilde\gamma_{6,11}\circ\widetilde\gamma_{11,6}<\id$, 
$\widetilde\gamma_{28,3}\circ\widetilde\gamma_{3,28}<\id$, 
$\widetilde\gamma_{4,1}\circ\widetilde\gamma_{1,29}\circ\widetilde\gamma_{29,4}<\id$, 
$\widetilde\gamma_{4,3}\circ\widetilde\gamma_{3,1}\circ\widetilde\gamma_{1,29}\circ\widetilde\gamma_{29,4}<\id$, 
$\widetilde\gamma_{3,1}\circ\widetilde\gamma_{1,29}\circ\widetilde\gamma_{29,28}\circ\widetilde\gamma_{28,3}<\id$, 
$\widetilde\gamma_{3,4}\circ\widetilde\gamma_{4,1}\circ\widetilde\gamma_{1,29}\circ\widetilde\gamma_{29,28}\circ\widetilde\gamma_{28,3}<\id$.

Let us verify it:\\
\begin{equation*}
\begin{array}{lllll}
\widetilde\gamma_{3,4}\circ\widetilde\gamma_{4,3}(t)&=&\frac{4}{5}(\widetilde\gamma_{4,3}(t))^2=\frac{4}{5}\sqrt{t}^2=\frac{4}{5}t&<&t\\
\widetilde\gamma_{3,4}\circ\widetilde\gamma_{4,6}\circ\widetilde\gamma_{6,3}(t)&=&\frac{4}{5}(\widetilde\gamma_{4,6}\circ\widetilde\gamma_{6,3}(t))^2=\frac{4}{5}(\widetilde\gamma_{6,3}(t))^2&&\\
&=&\frac{4}{5}(\sqrt{t})^2=\frac{4}{5}t&<&t;\\
\widetilde\gamma_{6,11}\circ\widetilde\gamma_{11,6}(t)&=&\widetilde\gamma_{11,6}(t)=\frac{2}{3}t&<&t;\\
\widetilde\gamma_{28,3}\circ\widetilde\gamma_{3,28}&=&\frac{13}{16}\widetilde\gamma_{3,28}(t)=\frac{13}{24}t&<&t;\\
\widetilde\gamma_{4,1}\circ\widetilde\gamma_{1,29}\circ\widetilde\gamma_{29,4}(t)&=&\widetilde\gamma_{1,29}\circ\widetilde\gamma_{29,4}(t)=\frac{4}{5}\cdot 1\cdot \frac{9}{10}t=\frac{18}{25}t&<&t;\\
\widetilde\gamma_{4,3}\circ\widetilde\gamma_{3,1}\circ\widetilde\gamma_{1,29}\circ\widetilde\gamma_{29,4}(t)&=&\sqrt{\widetilde\gamma_{3,1}\circ\widetilde\gamma_{1,29}\circ\widetilde\gamma_{29,4}(t)}&&\\
&=&\sqrt{\frac{5}{6}(\widetilde\gamma_{1,29}\circ\widetilde\gamma_{29,4})^2}=\sqrt{\frac{54}{125}}t&<&t;\\  
\widetilde\gamma_{3,1}\circ\widetilde\gamma_{1,29}\circ\widetilde\gamma_{29,28}\circ\widetilde\gamma_{28,3}(t)&=&\frac{5}{6}(\widetilde\gamma_{1,29}\circ\widetilde\gamma_{29,28}\circ\widetilde\gamma_{28,3}((t))^2&&\\
&=&\frac{5}{6}(\frac{4}{5}\widetilde\gamma_{29,28}\circ\widetilde\gamma_{28,3}(t))^2=\frac{5}{6}(\frac{4}{5}\sqrt{\frac{13}{16}t})^2=\frac{13}{30}t&<&t;\\
\widetilde\gamma_{3,4}\circ\widetilde\gamma_{4,1}\circ\widetilde\gamma_{1,29}\circ\widetilde\gamma_{29,28}\circ\widetilde\gamma_{28,3}(t)&=&\frac{4}{5}(\widetilde\gamma_{4,1}\circ\frac{4}{5}\sqrt{\frac{13}{16}t})^2=\frac{52}{125}t&<&t.
\end{array}
\end{equation*}

Thus by Corollary~\ref{thm:main} the large network is ISS.

Assume now that we know also ISS-Lyapunov functions $V_i$ for all subsystems of the large network. Let us construct an ISS-Lyapunov function for the small network and the for the large network. Consider the following functions $\sigma_i$: \\
$\widetilde{\sigma}_i(t)=\left\{
\begin{array}{ll}
t,&i\neq 2,6;\\
t^2, &i=2,6.
\end{array}\right.
$ 

Let us check whether $\widetilde\Gamma(\widetilde{\sigma})\leq\widetilde{\sigma}$:

$\begin{array}{lllll}
\tilde\gamma_{17}(\widetilde{\sigma}_7(t))&=&\frac{2}{3}t&\leq&t=\widetilde{\sigma}_1(t)\\
\max\{\tilde\gamma_{21}(\widetilde{\sigma}_1(t)),\tilde\gamma_{23}(\widetilde{\sigma}_3(t)),\tilde\gamma_{26}(\widetilde{\sigma}_6(t))\}&=&\max\{\frac{1}{2}
t^2,\frac{1}{2}t^2,\frac{1}{2}t^2\}&\leq&t^2=\widetilde{\sigma}_2(t)\\
\max\{\tilde\gamma_{31}(\widetilde{\sigma}_1(t)),\tilde\gamma_{32}(\widetilde{\sigma}_2(t)),\tilde\gamma_{34}(\widetilde{\sigma}_4(t))\}&=&
\max\{t,t,t\}&\leq&t=\widetilde{\sigma}_3(t)\\
\max\{\tilde\gamma_{42}(\widetilde{\sigma}_2(t)),\tilde\gamma_{45}(\widetilde{\sigma}_5(t))\}&=&
\max\{t,t\}&\leq&t=\widetilde{\sigma}_4(t)\\
\tilde\gamma_{54}(\widetilde{\sigma}_4(t))&=&\frac{2}{3}t&\leq&t=\widetilde{\sigma}_5(t)\\
\tilde\gamma_{62}(\widetilde{\sigma}_2(t))&=&\frac{3}{16}t^2&\leq&t^2=\widetilde{\sigma}_6(t)\\
\max\{\tilde\gamma_{73}(\widetilde{\sigma}_3(t)),\tilde\gamma_{76}(\widetilde{\sigma}_6(t))\}&=&\max\{\frac{1}{2}t,t\}&\leq&t=\widetilde{\sigma}_7(t)\\
\end{array}$

Thus $\widetilde\Gamma(\widetilde{\sigma})\leq\widetilde{\sigma}$ and ${\sigma}$ such that $\Gamma({\sigma})\leq{\sigma}$ can be constructed applying rules for sequentially connected nodes  in \eqref{omega_ser}, nodes connected in parallel in \eqref{omega_parallel} and almost disconnected subgraphs in \eqref{omega_subgraph}:\\
$\begin{array}{lllll}
\sigma_1(t)&=&\tilde\sigma_1(t)&=&t\\
\sigma_2(t)&=&\gamma_{2,13}\circ\gamma_{13,30}\circ\gamma_{30,29}\circ\tilde\sigma_{29}(t)&=&\frac{2}{5}t\\
\sigma_3(t)&=&\tilde\sigma_3(t)&=&t^2\\
\sigma_4(t)&=&\tilde\sigma_4(t)&=&t\\
\sigma_5(t)&=&\gamma_{5,3}\circ\tilde\sigma_3(t)&=&\frac{1}{2}t^2\\
\sigma_6(t)&=&\tilde\sigma_6(t)&=&t\\
\sigma_7(t)&=&\gamma_{7,4}\circ\tilde\sigma_4(t)&=&\sqrt{t}\\
\sigma_8(t)&=&\gamma_{8,5}\circ\sigma_5(t)&=&\frac{1}{\sqrt{2}}t\\
\sigma_9(t)&=&\gamma_{9,5}\circ\sigma_5(t)&=&\sqrt{2}t\\
\sigma_{10}(t)&=&\gamma_{10,5}\circ\sigma_5(t)&=&\frac{1}{2\sqrt{2}}t\\
\sigma_{11}(t)&=&\tilde\sigma_{11}(t)&=&t\\
\sigma_{12}(t)&=&\gamma_{12,7}\circ\sigma_7(t)&=&t\\
\sigma_{13}(t)&=&\gamma_{13,30}\circ\gamma_{30,29}\circ\tilde\sigma_29(t)&=&4t^2\\
\sigma_{14}(t)&=&\tilde\gamma_{14,5}\circ\sigma_5(t)&=&\frac{13}{8}t^2 \\
\sigma_{15}(t)&=&\max\{\gamma_{15,11},\gamma_{15,16}\circ\gamma_{16,11}\}\circ\sigma_{11}(t)&=&t^2\\
\sigma_{16}(t)&=&\max\{\gamma_{16,11},\gamma_{16,15}\circ\gamma_{15,11},\gamma_{16,22}\circ\gamma_{22,26}\circ\gamma_{26,20},&&\\ &&\circ\gamma_{20,15}\circ\gamma_{15,11}, \gamma_{16,22}\circ\gamma_{22,26}\circ\gamma_{26,19}, \circ\gamma_{19,15}\circ\gamma_{15,11}\}\circ\sigma_{11}(t)&=&t\\
\sigma_{17}(t)&=&\gamma_{17,12}\circ\sigma_{12}(t)&=&3t\\
\sigma_{18}(t)&=&\gamma_{18,14}\circ\sigma_{14}(t)&=&\frac{\sqrt{13}}{2\sqrt{2}}t \\ \sigma_{19}(t)&=&\max\{\gamma_{19,15}\circ\gamma_{15,11},\gamma_{19,15}\circ\gamma_{15,16}\circ\gamma_{16,11}\}\circ\sigma_{11}(t)&=&t  \\
\sigma_{20}(t)&=&\max\{\gamma_{20,15}\circ\gamma_{15,11},\gamma_{20,15}\circ\gamma_{15,16}\circ\gamma_{16,11}\}\circ\sigma_{11}(t)&=&2t\\
\sigma_{21}(t)&=&\max\{\gamma_{21,22}\circ\gamma_{22,26}\circ\gamma_{26,20}\circ\gamma_{20,15}\circ\gamma_{15,11}, \gamma_{21,22}\circ\gamma_{22,26}\circ\gamma_{26,19}\circ\gamma_{19,15}\circ\gamma_{15,11},&&\\ &&\gamma_{21,22}\circ\gamma_{22,26}\circ\gamma_{26,20}\circ\gamma_{20,15}\circ\gamma_{15,16}\circ\gamma_{16,11},&&\\ &&\gamma_{21,22}\circ\gamma_{22,26}\circ\gamma_{26,19}\circ\gamma_{19,15}\circ\gamma_{15,16}\circ\gamma_{16,11}\}\circ\sigma_{11}(t)&=&\frac{4}{9}t\\
\sigma_{22}(t)&=&\max\{\gamma_{22,26}\circ\gamma_{26,20}\circ\gamma_{20,15}\circ\gamma_{15,11}, \gamma_{22,26}\circ\gamma_{26,19}\circ\gamma_{19,15}\circ\gamma_{15,11},&&\\
 &&\gamma_{22,26}\circ\gamma_{26,20}\circ\gamma_{20,15}\circ\gamma_{15,16}\circ\gamma_{16,11},&&\\ &&\gamma_{22,26}\circ\gamma_{26,19}\circ\gamma_{19,15}\circ\gamma_{15,16}\circ\gamma_{16,11}\}\circ\sigma_{11}(t)&=&\frac{4}{81}t^2 \\
\sigma_{23}(t)&=&\gamma_{23,17}\circ\sigma_{17}(t)&=&\sqrt{3t}\\
\end{array}$
$\begin{array}{lllll}
\sigma_{24}(t)&=&\gamma_{24,18}\circ\sigma_{18}(t)&=&\frac{13}{8}t^2 \\
\sigma_{25}(t)&=&\gamma_{25,18}\circ\sigma_{18}(t)&=&\frac{13}{8}t^2\\
\sigma_{26}(t)&=&\max\{\gamma_{26,20}\circ\gamma_{20,15}\circ\gamma_{15,11}, \gamma_{26,19}\circ\gamma_{19,15}\circ\gamma_{15,11},&&\\ &&\gamma_{26,20}\circ\gamma_{20,15}\circ\gamma_{15,16}\circ\gamma_{16,11},&&\\ &&\gamma_{26,19}\circ\gamma_{19,15}\circ\gamma_{15,16}\circ\gamma_{16,11}\}\circ\sigma_{11}(t)&=&\frac{2}{3}t \\
\sigma_{27}(t)&=&\gamma_{27,25}\circ\sigma_{25}(t)&=&\frac{\sqrt{13}}{2\sqrt{2}}t\\
\sigma_{28}(t)&=&\tilde\sigma_{28}(t)&=&t^2\\
\sigma_{29}(t)&=&\tilde\sigma_{29}(t)&=&t \\
\sigma_{30}(t)&=&\gamma_{30,29}\circ\sigma_{29}(t)&=&4t^2.
\end{array}$

Then an ISS-Lyapunov function for the large network can be constructed as $V(x)=\max\limits_{i=1,\ldots,n}\sigma_i^{-1}(V_i(x_i))$ using Theorem~\ref{thm:iss_lyap}.

\section{Conclusions}\label{sec:Conclusion} 

The aggregation rules introduced in this paper preserve the main structure of a network and allow to reduce the number of computations during the verification of the small gain condition. Furthermore, in the case that there exist several motifs in one network, these rules can be applied step-by-step to reduce the size of the gain matrix $\Gamma$. The sequence of the application of these rules may be arbitrary or depend on some additional information about the network topology. For example, this sequence may depend on information about the most influential nodes of the network, see \cite[Algorithm~1]{SWD11}.

In this paper we  have performed only initial steps in the development of a structure-preserving reduction approach for large-scale networks with nonlinear dynamics. The next steps are: extension of the aggregation rules to other types of motifs, introduction and estimation of the error measure that compares the reduced and the original models, and the development of a numerical algorithm that performs this reduction. Further improvement of the approach may be performed by adaptation of the ranking technique used in \cite{SWD09} , \cite{SWD09b} 
and \cite{SWD11}.

\end{document}